\documentclass{article}
\usepackage[
    letterpaper,
    left=1.15in,
    right=1.15in,
    top=0.9in,
    bottom=0.95in
]{geometry}
\usepackage{graphicx} 

\usepackage{amsmath, amssymb, amsfonts, amsthm}
\usepackage{mathdots}
\usepackage{subcaption}
\usepackage{xcolor}
\usepackage{enumitem}
\usepackage[numbers]{natbib}

\usepackage{algorithm}
\usepackage{algpseudocode}

\usepackage{hyperref}
\usepackage{colortbl}
\definecolor{niceblue}{rgb}{0.0,0.19,0.56}

\usepackage{threeparttable}
\usepackage{pifont}
\newcommand{\cmark}{\ding{51}}
\newcommand{\xmark}{\ding{55}}
\newcommand{\algname}[1]{\textsc{#1}}
\definecolor{bgcolor2}{RGB}{235,245,255}

\usepackage[capitalise]{cleveref}

\usepackage{todonotes}
\makeatletter
\renewcommand*\env@matrix[1][*\c@MaxMatrixCols c]{%
  \hskip -\arraycolsep
  \let\@ifnextchar\new@ifnextchar
  \array{#1}}
\makeatother

\usepackage{array,booktabs,multirow,makecell}

\usepackage{nicefrac}

\allowdisplaybreaks

\title{Direct Acceleration of Stochastic Root-Finding\\ Without Variance Reduction and Regularization}

\author{TaeHo Yoon\textsuperscript{*}
\and Nicolas Loizou\textsuperscript{*}
}

\date{}

\newcommand{\cB}{{\mathcal{B}}}

\newcommand{\cF}{{\mathcal{F}}}
\newcommand{\cG}{{\mathcal{G}}}

\newcommand{\cN}{{\mathcal{N}}}
\newcommand{\cO}{{\mathcal{O}}}

\newcommand{\cQ}{{\mathcal{Q}}}

\usepackage[bb=boondox]{mathalfa}

\usepackage{xparse}
\DeclareFontFamily{U}{ntxmia}{}
\DeclareFontShape{U}{ntxmia}{m}{it}{<-> ntxmia }{}
\DeclareFontShape{U}{ntxmia}{b}{it}{<-> ntxbmia }{}
\DeclareSymbolFont{lettersA}{U}{ntxmia}{m}{it}
\SetSymbolFont{lettersA}{bold}{U}{ntxmia}{b}{it}

\ExplSyntaxOn
\NewDocumentCommand{\varmathbb}{m}
 {
  \tl_map_inline:nn { #1 }
   {
    \use:c { varbb##1 }
   }
 }
\tl_map_inline:nn { ABCDEFGHIJKLMNOPQRSTUVWXYZ }
 {
  \exp_args:Nc \DeclareMathSymbol{varbb#1}{\mathord}{lettersA}{\int_eval:n { `#1+67 }}
 }
\exp_args:Nc \DeclareMathSymbol{varbbk}{\mathord}{lettersA}{169}
\ExplSyntaxOff

\newcommand*{\Fix}{\mathrm{Fix}\,}
\newcommand*{\Zer}{\mathrm{Zer}\,}

\newcommand{\reals}{\mathbb{R}}

\newcommand{\opF}{{\varmathbb{F}}}

\newcommand{\opH}{{\varmathbb{H}}}
\newcommand{\opI}{{\varmathbb{I}}}

\newcommand{\opT}{{\varmathbb{T}}}

\newcommand{\inprod}[2]{\left\langle #1,#2 \right\rangle}

\newcommand{\sqnorm}[1]{\left\| #1 \right\|^2}
\newcommand{\norm}[1]{\left\| #1 \right\|}

\newcommand{\expec}[1]{\mathbb{E}\left[ #1 \right]}
\newcommand{\condexp}[2]{\mathbb{E}\left[ #1 \,\middle|\, #2 \right]}

\newcommand{\R}{\mathbb{R}}
\newcommand{\E}{\mathbb{E}}

\usepackage{thmtools}
\definecolor{shadecolor}{gray}{0.9}
\declaretheoremstyle[
headfont=\normalfont\bfseries,
notefont=\mdseries, notebraces={(}{)},
bodyfont=\normalfont,
postheadspace=0.5em,
spaceabove=\topsep,
mdframed={
  skipabove=8pt,
  skipbelow=8pt,
  hidealllines=true,
  backgroundcolor={shadecolor},
  innerleftmargin=4pt,
  innerrightmargin=4pt}
]{shaded}

\declaretheorem[style=shaded,within=section]{definition}
\declaretheorem[style=shaded,sibling=definition]{theorem}
\declaretheorem[style=shaded,sibling=definition]{proposition}
\declaretheorem[style=shaded,sibling=definition]{assumption}

\declaretheorem[style=shaded,sibling=definition]{lemma}

\begin{document}

\renewcommand\thefootnote{\fnsymbol{footnote}}
\footnotetext[1]{Department of Applied Mathematics and Statistics \& Mathematical Institute for Data Science, Johns Hopkins University.
\texttt{\{tyoon7, nloizou\}@jhu.edu}}

\renewcommand\thefootnote{\arabic{footnote}}
\setcounter{footnote}{0}

\maketitle

\begin{abstract}
Acceleration for deterministic root-finding problems has been extensively studied in recent years; specifically, the anchor-based, or Halpern-type methods achieve optimal convergence rates with respect to the operator norm.
However, acceleration via these methods does not directly carry over to stochastic setting due to accumulation of errors, unless one enforces diminishing variance via increasing batch sizes or variance reduction techniques.
In this work, we show that another class of acceleration, namely the \textit{dual-anchor} mechanism, extends to the stochastic setting without such error accumulation, in contrast to anchor-based algorithms.
Consequently, we cleanly achieve $\cO(\epsilon^{-3})$ complexity with iteration-independent batch size, without any variance reduction or double-loop recursive regularization, for stochastic root-finding (\textit{resp.}\ fixed-point) problems with cocoercivity (\textit{resp.}\ square-nonexpansivity) in expectation.
For strongly monotone operators, the same algorithm attains a sharper $\widetilde\cO(\epsilon^{-2})$ complexity, nearly matching the lower bound in terms of $\epsilon$-dependence.
\end{abstract}

\section{Introduction}

Root-finding problems for an operator $\opF\colon \reals^d \to \reals^d$, stated as
\begin{align}
\label{eqn:root-finding-problem}
    \underset{x\in \reals^d}{\text{find}} \quad \opF(x) = 0 ,
\end{align}
generalize minimization problems and have been studied under distinct formulations including monotone inclusion \citep{Minty1962_monotone, rockafellarMonotoneOperatorsProximal1976, BauschkeCombettes2017_convex, RyuYin2022_largescale}, minimax optimization \citep{nemirovskiProxmethodRateConvergence2004, goodfellowGenerativeAdversarialNets2014, MadryMakelovSchmidtTsiprasVladu2018_deep, DaskalakisIlyasSyrgkanisZeng2018_training, mertikopoulosOptimisticMirrorDescent2019, LinJinJordan2020_nearoptimal}, equilibrium search \citep{stampacchiaFormesBilineairesCoercitives1964, FacchineiPang2003_finitedimensional, scutariRealComplexMonotone2014}, or fixed-point problems \citep{Krasnoselskii1955_two, Mann1953_mean, Halpern1967_fixed, Wittmann1992_approximation, SabachShtern2017_first}.
In the past few years, there have been significant developments in accelerated algorithms for reducing the squared operator/residual norm with the optimal $\sqnorm{\opF(x_k)} = \cO(1/k^2)$ rate \citep{Lieder2021_convergence, Diakonikolas2020_halpern, Kim2021_accelerated, ParkRyu2022_exact, yoonAcceleratedAlgorithmsSmooth2021, LeeKim2021_fast}. 
A recurring message from the optimization literature, however, is that deterministic acceleration is often \emph{brittle} under oracle errors; in the stochastic setting where $\opF(x) = \expec{\widehat \opF(x;\xi)}$ and the algorithm uses $\widehat \opF(x;\xi)$, error may accumulate quickly with these accelerated methods and break their fast convergence from the noiseless setting.
Consequently, for stochastic optimization problems, taking advantage of the study of acceleration was not straightforward---it required decaying oracle inexactness, variance reduction, restart mechanisms, or stronger assumptions on the problem \citep{devolderFirstorderMethodsSmooth2014, ghadimiAcceleratedGradientMethods2016, liPAGESimpleOptimal2021, cai2022stochastic}.
Specifically for stochastic monotone inclusion problems, stochastic Halpern iteration and accelerated (anchored) extragradient type analyses suffered precisely from these issues \citep{LeeKim2021_fast, cai2022stochastic, caiVarianceReducedHalpern2024}.

We take a different viewpoint toward handling stochastic instability. 
Recent work \citep{yoonOptimalAccelerationMinimax2024} showed that the acceleration mechanism in fixed-point and minimax problems is not unique: besides anchor-based algorithms \citep{Lieder2021_convergence, Kim2021_accelerated, yoonAcceleratedAlgorithmsSmooth2021, LeeKim2021_fast, Tran-DinhLuo2021_halperntype}, there exist their \textit{H-dual} algorithms, which display materially different update rules but share exactly the same deterministic last-iterate rate with the optimal anchor-based acceleration. 
Hence, the deterministic worst-case complexity alone does not identify a single preferred accelerated method.
In this paper, we demonstrate that the stochastic extension can break this tie: the \textit{Stochastic Dual Optimal Halpern Method (S-Dual-OHM)} attains $\expec{\norm{\opF(x_{N-1})}} \le \epsilon$ with $\cO(\epsilon^{-3})$ oracle complexity, with iteration complexity $N=\cO(\epsilon^{-1})$ and with constant mini-batch size $\cO(\epsilon^{-2})$.
This is done cleanly and more directly without assuming diminishing variance or using variance reduction techniques, which is not what stochastic anchoring methods were capable of \citep{LeeKim2021_fast, cai2022stochastic, caiVarianceReducedHalpern2024}.

In our view, this result has a conceptual implication beyond simply providing another efficient stochastic algorithm.
It challenges the folklore view that acceleration, as studied in deterministic optimization, is brittle and incompatible with stochastic noise.
Equivalently, optimal deterministic methods can behave very differently under stochastic perturbations, and selecting the right representation of acceleration may enable the design of efficient stochastic algorithms.

\paragraph{Contributions.}
We summarize our main contributions as follows.

\begin{table*}[t]
    \centering
    \caption{\small Comparison of stochastic algorithms for finding an $\epsilon$-approximate solution with $\mathbb{E}\!\left[\|\opF(x)\|\right] \le \epsilon$ for $\nicefrac{1}{L}$-cocoercive operator $\opF$.}
    \vspace{1mm}
    \label{table:algorithm-comparison}
    \begin{threeparttable}
    \begingroup
    \setlength{\tabcolsep}{5.2pt}
    \renewcommand{\arraystretch}{1.12}
    \resizebox{\textwidth}{!}{%
    \begin{tabular}{lcccccc}
        \toprule
        Algorithms
        & Complexity
        & \makecell{Last-\\iterate}
        & \makecell{Variance-\\reduction free}
        & \makecell{Single-\\loop}
        & \makecell{Algorithm\\parameters}
        & \makecell{Theory\\requirements} 
        \\
        \midrule
        \algname{SGDA}/\algname{SEG} \citep{gorbunovStochasticExtragradientGeneral2022}
        & $\cO(\epsilon^{-4})$\textsuperscript{$\dagger$}
        & \xmark
        & \cmark
        & \cmark
        & $L,B$
        & $D,\sigma$
        \\
        \addlinespace[2pt]
        \algname{S-Halpern-PAGE} \citep{cai2022stochastic}
        & $\cO(\epsilon^{-3})$
        & \cmark
        & \xmark
        & \cmark
        & $L,p_k,S_1^{(k)}, S_2^{(k)}$
        & $D,\sigma$
        \\
        \addlinespace[2pt]
        \algname{S-Halpern-VR-Finite} \citep{caiVarianceReducedHalpern2024}
        & $\widetilde{\cO}(n + \sqrt{n}\,\epsilon^{-1})$\textsuperscript{$\ddagger$}
        & \cmark
        & \xmark
        & \cmark
        & $L,B,\lambda_k; \, p_k$ or $M_k$
        & $D,n$
        \\
        \addlinespace[2pt]
        \algname{RAIN} \citep{chenNearoptimalAlgorithmsMaking2024}
        & $\widetilde{\cO}(\epsilon^{-2})$
        & \cmark
        & \cmark
        & \xmark\textsuperscript{*}
        & $L,\lambda, \gamma, N_s, K_s$
        & $D,\sigma$
        \\
        \addlinespace[2pt]
        \rowcolor{bgcolor2}
        \algname{S-Dual-OHM} (This work)
        & $\cO(\epsilon^{-3})$
        & \cmark
        & \cmark
        & \cmark
        & $L,B,N$
        & $D,\sigma$
        \\
        \bottomrule
    \end{tabular}%
    }
    \endgroup
    \begin{tablenotes}[flushleft]
        \scriptsize
        \item[]\begin{minipage}[t]{0.99\textwidth}
        \algname{SGDA} denotes Stochastic Gradient Descent-Ascent $x_{k+1} = x_k - \alpha \opF(x_k; \xi_k)$ and \algname{SEG} denotes stochastic version of Extragradient \citep{Korpelevich1976_extragradient}.
        ``Last-iterate'' indicates whether the stated guarantee is for the final iterate and not the averaged, best, or uniform random iterate. ``Variance-reduction free'' indicates whether the algorithm avoids using variance reduction techniques. ``Single-loop'' indicates whether the algorithm avoids using outer--inner loop structure (e.g., regularization techniques). 
        ``Algorithm parameters'' lists the hyperparemeters whose values should be specified for the run, 
        where the cocoercivity/Lipschitzness parameter $L$ is used for determining the step-sizes of the algorithm.
        ``Theory requirements'' lists the information of problem-related parameters that are needed for selecting the hyperparameters that theoretically guarantee $\mathbb{E}\!\left[\|\opF(x)\|\right] \le \epsilon$.
        Here $N$ is the total number of algorithm iterations, $D = \norm{x_0 - x_\star}$, $\sigma$ is a bound on the operator variance, and $n$ is the number of component operators in the case where the problem has a finite-sum structure.
        The $\widetilde{\cO}(\cdot)$ notation hides logarithmic factors.\\
        \textsuperscript{$\dagger$}\citep[Corollary~E.4]{gorbunovStochasticExtragradientGeneral2022} shows this for uniform random iterate of \algname{SEG}, and a similar proof for \algname{SGDA} can be derived.\\
        \textsuperscript{$\ddagger$}Works only for finite-sum problems where $\opF(\cdot)=\frac{1}{n}\sum_{i=1}^n \opF(\cdot;\xi_i)$.\\
        \textsuperscript{*}\citep{chenNearoptimalAlgorithmsMaking2024} presents a single-loop version of \algname{RAIN} for numerical experiments (which we use in our experiments for comparison), but their theoretical convergence analysis still requires a nontrivial double-loop structure.
        \end{minipage}
        \vspace{-2mm}
    \end{tablenotes}
    \end{threeparttable}
\end{table*}

\begin{itemize}[leftmargin=1.2em,itemsep=0.25em]
    \item[$\diamond$] 
    We show that under cocoercivity of $\opF$ in expectation, our proposed \ref{eqn:stochastic-Dual-OHM} with mini-batching over $\widehat \opF(\cdot;\xi)$ with \textit{constant step-size} $\alpha \in \left(0, \frac{2}{L}\right]$ and \textit{constant batch size} $B = \cO(\epsilon^{-2})$ achieves the bound $\expec{\norm{\opF(x_{N-1})}} \le \epsilon$ within $N = \cO(\epsilon^{-1})$ iterations.
    The resulting total oracle complexity $\cO(\epsilon^{-3})$, to the best of our knowledge, has not been achieved in prior work without variance reduction or double-loop regularization techniques (\cref{table:algorithm-comparison}). 

    \item[$\diamond$] We show that when $\opF$ is additionally strongly monotone, the same \ref{eqn:stochastic-Dual-OHM} algorithm can be early-stopped in $k = \widetilde\cO(\nicefrac{L}{\mu}\log\epsilon^{-1})$ iterations to attain $\expec{\norm{\opF(x_k)}} \le \epsilon$.
    This yields the oracle complexity $\widetilde\cO(\epsilon^{-2})$, which has near-optimal dependence on $\epsilon$.

    \item[$\diamond$] We provide numerical experiments demonstrating that Dual-OHM is indeed more robust to stochastic noise than Halpern iteration, and that \ref{eqn:stochastic-Dual-OHM} can effectively solve stochastic root-finding problems.
    
\end{itemize}

\section{Related Work}
\label{section:related-work}

\paragraph{Fixed-point problems, Halpern iteration and acceleration.}
For fixed-point problems
\begin{align}
\label{eqn:fixed-point-problem}
    \underset{x\in \reals^d}{\text{find}} \quad x = \opT(x) ,
\end{align}
with a nonexpansive operator $\opT\colon \reals^d \to \reals^d$, the Halpern iteration $x_{k+1} = \beta_k x_0 + (1-\beta_k) \opT x_k$ has been studied over a long history in the literature \citep{Halpern1967_fixed, Wittmann1992_approximation}, and has been found to provide accelerated worst-case convergence rate $\norm{x_k - \opT(x_k)} = \cO(1/k)$ with $\beta_k = \cO(1/k)$ \citep{SabachShtern2017_first, Diakonikolas2020_halpern, Lieder2021_convergence, Kim2021_accelerated, LeeRyu2023_accelerating}. 
The best choice $\beta_k = \frac{1}{k+2}$, due to \citep{Lieder2021_convergence, Kim2021_accelerated}, was later shown to yield an optimal rate $\sqnorm{x_k - \opT(x_k)} \le \frac{4\sqnorm{x_0 - x_\star}}{(k+1)^2}$, exactly matching the lower bound \citep{ParkRyu2022_exact}.
More recently, it was discovered that the same exact optimal rate for a fixed finite horizon $N$ can be achieved by distinct types of acceleration forming an infinite family of algorithms \citep{yoonHinvarianceTheoryComplete2025} including Dual-OHM \citep{yoonOptimalAccelerationMinimax2024}, the basis of this work.

\paragraph{Acceleration of minimax optimization and monotone inclusion.}
Minimax optimization with convex-concave objective can be recast into monotone inclusion, i.e., root-finding problem~\eqref{eqn:root-finding-problem} with saddle gradient operator $\opF$, which is a monotone operator \citep{rockafellarMonotoneOperatorsProximal1976, FacchineiPang2003_finitedimensional}.
Some early works including \citep{RyuYuanYin2019_ode, Diakonikolas2020_halpern} first connected Halpern acceleration to these problems, and \citep{Diakonikolas2020_halpern} achieved near-optimal complexity for Lipschitz continuous monotone inclusions. 
The remaining logarithmic gap was removed by the Extra Anchored Gradient algorithm \citep{yoonAcceleratedAlgorithmsSmooth2021}, and the so-called anchor acceleration for minimax optimization and monotone inclusion has been studied extensively since then \citep{LeeKim2021_fast, Tran-DinhLuo2021_halperntype, yoonAcceleratedMinimaxAlgorithms2025, SuhParkRyu2023_continuoustime, CaiZheng2023_accelerated, botExtragradientMethodFlexible2026, alcalaMovingAnchorExtragradient2023, leeNearoptimalSampleComplexity2025}.
Several works proposed its extension via distinct interpretation in connection with Nesterov momentum \citep{BotCsetnekNguyen2023_fast, tran-dinhHalpernsFixedpointIterations2024}.
On the other hand, the dual-anchor acceleration analogous to Dual-OHM from fixed-point problems is achievable in this setting as well \citep{yoonOptimalAccelerationMinimax2024}.

\paragraph{Residual-reducing algorithms for stochastic fixed-point and monotone inclusion problems.}
The Fast Extragradient algorithm \citep{LeeKim2021_fast}, an anchor-based accelerated minimax algorithm, was proposed with its stochastic extension under the assumption that the operator's variance at $k$-th iteration diminishes at the order of $\cO(1/k)$.
Then \citep{cai2022stochastic} showed that stochastic Halpern iteration and its minimax variants, combined with PAGE variance reduction \citep{liPAGESimpleOptimal2021} can reduce the complexity of achieving $\expec{\norm{\opF(x_k)}} \le \epsilon$ to $\cO(\epsilon^{-3})$, which could be further improved to $\widetilde \cO(n + \sqrt{n}\epsilon^{-1})$ \citep{caiVarianceReducedHalpern2024} in the finite-sum setting where $\opF$ is the average of $n=o(\epsilon^{-3})$ sample operators.
Near-optimal complexity of $\widetilde \cO(\epsilon^{-2})$ for stochastic monotone inclusion has been achieved in \citep{chenNearoptimalAlgorithmsMaking2024} with the matching $\Omega(\epsilon^{-2})$ lower bound via more sophisticated recursive anchoring (regularization) algorithm.
On another line of work, \citep{bravoStochasticFixedpointIterations2024, pischkeAsymptoticRegularityGeneralised2026, bravoStochasticHalpernIteration2026} studied stochastic Krasnosel'ski\u{\i}--Mann or Halpern iterations in general normed spaces, where the lower bound of $\Omega(\epsilon^{-3})$ was established \citep{bravoStochasticHalpernIteration2026}.

\paragraph{Broader view of extending acceleration to stochastic settings.}
In smooth convex minimization, it is now standard that accelerated methods are fragile to persistent oracle error: preserving acceleration under stochasticity typically requires either variance reduction or catalyst-type outer--inner regularization \citep{devolderFirstorderMethodsSmooth2014, linUniversalCatalystFirstorder2015, allen-zhuKatyushaFirstDirect2018}. 
A similar viewpoint has been adopted for monotone root-finding problems in prior works \citep{cai2022stochastic, caiVarianceReducedHalpern2024, chenNearoptimalAlgorithmsMaking2024}.
Our approach departs from this line of work.
Rather than stabilizing Halpern-type acceleration by additional variance-reduction or regularization, we leverage the fact that deterministic acceleration is not unique, and there exist distinct optimal accelerated algorithms with the same worst-case deterministic rate. 
We deliver the new observation that among them, Dual-OHM is intrinsically less sensitive to stochastic perturbations than Halpern-type algorithms, enabling direct stochastic extension.

\section{Preliminaries and Assumptions}

In this section, we establish some necessary preliminary concepts for subsequent analysis.

\subsection{Monotonicity, cocoercivity, and nonexpansiveness of operators}

We say that an operator $\opF\colon \reals^d \to \reals^d$ is
\begin{itemize}[leftmargin=1.2em,itemsep=0.25em]
    \item[-] \textit{Monotone} if $\inprod{\opF (x) - \opF (y)}{x-y} \ge 0$ for all $x,y \in \reals^d$ 
    \item[-] $\mu$-\textit{strongly monotone} if $\inprod{\opF (x) - \opF (y)}{x-y} \ge \mu \sqnorm{x-y}$ for all $x,y \in \reals^d$
    \item[-] $M$-\textit{Lipschitz} if $\norm{\opF (x) - \opF (y)} \le M\norm{x-y}$ for all $x,y \in \reals^d$ 
    \item[-] $\nicefrac{1}{L}$-\textit{Cocoercive} if $\inprod{\opF (x) - \opF (y)}{x-y} \ge \frac{1}{L} \sqnorm{ \opF (x) - \opF (y)}$ for all $x,y \in \reals^d$ 
\end{itemize}
where $\mu,M,L>0$.
In particular, $\nicefrac{1}{L}$-cocoercivity implies monotonicity and $L$-Lipschitzness, but not vice versa.
If $\opF$ is $M$-Lipschitz and $\mu$-strongly monotone, then it is $\nicefrac{1}{L}$-cocoercive with $L=\nicefrac{M^2}{\mu}$ \citep{FacchineiPang2003_finitedimensional}.
When $\opF$ is $\nicefrac{1}{L}$-cocoercive, the problem~\eqref{eqn:root-finding-problem} can be recast into a fixed-point problem~\eqref{eqn:fixed-point-problem} where $\opT = \opI - \alpha \opF \colon \R^d\to\R^d$ is \textit{nonexpansive}, i.e., $1$-Lipschitz for $0<\alpha \le \frac{2}{L}$ \citep[Proposition~4.39]{BauschkeCombettes2017_convex}, where $\opI \colon \R^d\to\R^d$ is the identity operator.
This is because 
\[
    \Zer(\opF) := \{x \in\reals^d \,|\, \opF (x) = 0\} = \{x\in\reals^d \,|\, \opT (x) = x \} := \Fix(\opT) .
\]

Given an unconstrained minimax problem 
\begin{align}
    \underset{u\in \reals^m}{\text{minimize}} \,\, \underset{v\in \reals^n}{\text{maximize}} \,\, \Phi(u,v)
\end{align}
with a convex-concave (\textit{resp.}\ $\mu$-strongly-convex-strongly-concave) and $M$-smooth $\Phi\colon \reals^m \times \reals^n \to \reals$, its saddle operator
\[
    \opF(u,v) = (\nabla_u \Phi(u,v), -\nabla_v \Phi(u,v))
\]
is monotone (\textit{resp.}\ $\mu$-strongly monotone) and $M$-Lipschitz \citep{FacchineiPang2003_finitedimensional}.

\subsection{Stochastic oracle models}

We consider an unbiased stochastic operator oracle with bounded variance.

\begin{assumption}[Stochastic operator oracle]
\label{assumption:stochastic-operator-oracle}
Given $x\in \R^d$, one can query $\widehat \opF(x;\xi)$ such that
\[
\expec{\widehat \opF(x;\xi)} = \opF(x) ,
\qquad
\expec{\sqnorm{\widehat \opF(x;\xi) - \opF(x)} } \le \sigma^2.
\]
\end{assumption}

For our main result (\cref{theorem:stochastic-Dual-OHM-cocoercive}), we additionally assume the following condition.

\begin{assumption}[Cocoercivity in expectation]
\label{assumption:cocoercivity-in-expectation}
For any $x,y\in \reals^d$, we have
\[
    \inprod{\opF(x) - \opF(y)}{x-y} \ge \mathbb{E} \left[ \frac{1}{L} \sqnorm{\widehat \opF(x;\xi) - \widehat \opF(y;\xi)} \right] .
\]
\end{assumption}

This condition holds for some natural cases, including: \textbf{(i)} stochastic oracle with additive noise $\widehat{\opF}(x;\xi)=\opF(x) + \zeta(\xi)$, where the noise $\zeta(\xi)$ is independent of $x$, \textbf{(ii)} finite sum or expectation models where each sample operator is $\nicefrac{1}{L}$-cocoercive, and as its special case, \textbf{(iii)} sample operators that are each $\mu$-strongly monotone ($\mu > 0$) and $M$-Lipschitz, which implies sample-wise cocoercivity \cite{FacchineiPang2003_finitedimensional}. 
Note that \cref{assumption:cocoercivity-in-expectation} does not capture general monotone and Lipschitz sample operators or the cases where the noise $\zeta(\xi)$ depends on $x$.
Similar assumptions or even stronger samplewise cocoercivity have been commonly used in stochastic optimization \cite{nguyenSARAHNovelMethod2017,allen-zhuKatyushaFirstDirect2018,LoizouBerardGidelMitliagkasLacoste-Julien2021_stochastic,morinCocoercivitySmoothnessBias2022,BeznosikovGorbunovBerardLoizou2023_stochastic,zhang2024communication, yoon2026multiplayer,davisVarianceReductionRootfinding2023,caiVarianceReducedHalpern2024}.

Given batch size $B\ge 1$, we denote $\opF_\cB(x) = \frac{1}{B} \sum_{b=1}^B \widehat \opF(x;\xi_b)$, where $\xi_b$ are conditionally independent samples given $x$ and $\cB=\{\xi_1,\dots,\xi_B\}$.
Then $\opF_\cB(x)$ is an unbiased estimator of $\opF(x)$ with $\expec{\sqnorm{\opF_\cB (x) - \opF(x)}} \le \frac{\sigma^2}{B}$ and $\inprod{\opF(x) - \opF(y)}{x-y} \ge \mathbb{E} \left[ \frac{1}{L} \sqnorm{\opF_\cB (x) - \opF_\cB (y)} \right]$ by Jensen's inequality.
For $0<\alpha \le \frac{2}{L}$, if we let $\widehat \opT(\cdot;\xi) = \opI - \alpha \widehat \opF(\cdot; \xi)$ and $\opT_\cB(\cdot) = \frac{1}{B} \sum_{b=1}^B \widehat \opT(\cdot; \xi_b)$, then these are unbiased estimators of the nonexpansive $\opT = \opI - \alpha\opF$ with variance bounds $\alpha^2 \sigma^2$ and $\frac{\alpha^2 \sigma^2}{B}$, respectively.
Furthermore, we have: $\expec{\sqnorm{\opT_\cB (x) - \opT_\cB (y)}} \le \sqnorm{x-y}$ (square-nonexpansivity in expectation).

\subsection{Known acceleration in fixed-point problems and its extension}

Our general strategy is to reformulate \eqref{eqn:root-finding-problem} with cocoercive $\opF$ into \eqref{eqn:fixed-point-problem} with $\opT = \opI - \alpha\opF$ and build upon deterministic acceleration results.
Consider the following algorithms: given $y_0 \in \reals^d$,
\begin{align}
\label{eqn:OHM}
\tag{OHM}
    y_{k+1} = \frac{1}{k+2} y_0 + \frac{k+1}{k+2} \opT y_k 
\end{align}
and, with the convention $\opT y_{-1} = y_0$,
\begin{align}
\label{eqn:Dual-OHM}
\tag{Dual-OHM}
    y_{k+1} = y_k + \frac{N-k-1}{N-k} \left( \opT y_k - \opT y_{k-1} \right)
\end{align}
Here, \ref{eqn:OHM} is an anytime algorithm whose update is defined for all $k=0,1,\dots$, while for \ref{eqn:Dual-OHM}, the terminal iteration number $N$ is fixed, and the update is defined only for $k=0,\dots,N-2$.
This is a substantive distinction that makes the two algorithms behave differently under stochastic extension (see the discussion in Section~\ref{section:intuitive-explanation} for further details).

\begin{proposition}[Deterministic acceleration \citep{Lieder2021_convergence, Kim2021_accelerated, yoonOptimalAccelerationMinimax2024}]
\label{proposition:deterministic-acceleration}
For a nonexpansive operator $\opT\colon \reals^d \to \reals^d$ with a fixed point $y_\star \in \Fix\opT$ and $N \ge 2$, both \ref{eqn:OHM} and \ref{eqn:Dual-OHM} exhibit the rate
\[
\norm{y_{N-1} - \opT (y_{N-1})}^2 \le \frac{4\norm{y_0-y_\star}^2}{N^2} .
\]
\end{proposition}

With $\opT = \opI - \alpha\opF$ and $\alpha = \frac{1}{L}$, \cref{proposition:deterministic-acceleration} shows that $\norm{\opF(y_{N-1})} = \frac{\norm{y_{N-1} - \opT(y_{N-1})}}{\alpha} \le \epsilon$ is achieved in $\cO\left( \frac{LD}{\epsilon} \right)$ iterations, where $D = \norm{y_0 - y_\star}$.

\begin{proposition}[\citet{cai2022stochastic}, Informal]
\label{proposition:stochastic-OHM-prior-result}
Stochastic OHM (\algname{S-OHM}), which has the update rule $x_{k+1} = \frac{1}{k+2}x_0 + \frac{k+1}{k+2}\opT_{\cB_k}(x_k)$,  maintains the deterministic iteration bound $N=\cO(L\norm{x_0 - x_\star}/\epsilon)$ if $\E \norm{\opF_{\cB_k} (x_k) - \opF (x_k)}^2 \lesssim \frac{\epsilon^2}{k}$. 
Simple mini-batching with $|\cB_k| = \Theta(k\epsilon^{-2})$ yields $\cO(\epsilon^{-4})$ stochastic oracle complexity, and using PAGE variance reduction, this can be reduced to $\cO(\epsilon^{-3})$.
\end{proposition}

\section{Stochastic Dual-Anchor Acceleration}

In this section, we analyze \textit{Stochastic Dual-OHM (S-Dual-OHM)}: with $\opT_{\cB_{-1}}(x_{-1}) = x_0$,
\begin{align}
\label{eqn:stochastic-Dual-OHM}
\tag{S-Dual-OHM}
    x_{k+1} = x_k + \frac{N-k-1}{N-k} \left(\opT_{\cB_k} (x_k) - \opT_{\cB_{k-1}} (x_{k-1}) \right) 
\end{align}
for $k=0,\dots,N-2$, and present the following main result of this paper.

\begin{theorem}
\label{theorem:stochastic-Dual-OHM-cocoercive}
Consider the root-finding problem~\eqref{eqn:root-finding-problem} with a solution $x_\star$, satisfying Assumptions~\ref{assumption:stochastic-operator-oracle} and \ref{assumption:cocoercivity-in-expectation}. 
Then for $0<\alpha\le\frac{2}{L}$, \ref{eqn:stochastic-Dual-OHM} with $\opT_{\cB_k} = \opI - \alpha\opF_{\cB_k}$ with $|\cB_k| \equiv B$ satisfies
\begin{align*}
    \expec{\norm{\opF(x_{N-1})}}^2 \le \expec{\sqnorm{\opF(x_{N-1})}} \le \frac{4\sqnorm{x_0 - x_\star}}{\alpha^2 N^2} + \frac{6\sigma^2}{B} .
\end{align*}
\end{theorem}

This shows that we will have $\expec{\norm{\opF(x_{N-1})}} \le \epsilon$ with $N = \Theta(1/\epsilon)$ and $B = \Theta(1/\epsilon^2)$, so in total, $\cO(1/\epsilon^3)$ oracle complexity.
This also proves an equivalent bound on the expected fixed-point residual $\expec{\norm{x_{N-1} - \opT (x_{N-1})}}^2 \le \frac{4\sqnorm{x_0 - x_\star}}{N^2} + \frac{6\alpha^2 \sigma^2}{B}$.
While we defer technical details to Appendix~\ref{section:appendix-missing-proofs}, we provide the outline of key arguments below.
Then, we show that with the additional assumption of strong monotonicity of $\opF$, the total complexity can be reduced to $\widetilde \cO(1/\epsilon^2)$ with early stopping.

\subsection{Proof outline for \texorpdfstring{\cref{theorem:stochastic-Dual-OHM-cocoercive}}{Theorem 4.1}}

We start with showing the identity characterizing the S-Dual-OHM algorithm.

\begin{lemma}
\label{lemma:stochastic-Dual-OHM-magic-identity}
For \ref{eqn:stochastic-Dual-OHM} with $\opT_{\cB_k} = \opI - \alpha\opF_{\cB_k}$, the following identity holds:
\begin{align*}
    0 & = \frac{(N-1)\alpha^2}{4} \sqnorm{\opF (x_{N-1})} + \frac{\alpha}{2} \inprod{\opF (x_{N-1})}{x_{N-1} - x_0} + \sum_{j=1}^{N-1} \lambda_{N,j} \cQ_{N,j}
\end{align*}
where for $j=1,\dots,N-1$, $\lambda_{N,j} = \frac{N\alpha}{2(N-j)(N-j+1)}$ and
\[
    \cQ_{N,j} = \inprod{x_{j-1} - x_{N-1}}{\opF_{\cB_{j-1}}(x_{j-1}) - \opF(x_{N-1})} - \frac{\alpha}{2} \sqnorm{\opF_{\cB_{j-1}} (x_{j-1}) - \opF(x_{N-1})} .
\] 
\end{lemma}

Here, when the problem is deterministic so that $\opF_{\cB_k} = \opF$ for all $k=0,\dots,N-1$ and $0<\alpha\le\frac{2}{L}$, then by cocoercivity $\cQ_{N,j} \ge 0$ for all $j=1,\dots,N-1$, so $\sum_{j=1}^{N-1} \lambda_{N,j} \cQ_{N,j} \ge 0$. 
The resulting inequality, together with Young's inequality and one more application of cocoercivity: $\inprod{\opF(x_{N-1})}{x_{N-1}-x_\star} \ge \frac{1}{L} \sqnorm{\opF(x_{N-1})} \ge \frac{\alpha}{2} \sqnorm{\opF(x_{N-1})}$, yields
\begin{align}
\label{eqn:deterministic-case-final-bound}
    0 & \ge \frac{(N-1)\alpha^2}{4} \sqnorm{\opF (x_{N-1})} + \frac{\alpha}{2} \inprod{\opF (x_{N-1})}{x_{N-1} - x_0} \\
    & \ge \frac{N\alpha^2}{4} \sqnorm{\opF (x_{N-1})} - \frac{N\alpha^2}{8} \sqnorm{\opF (x_{N-1})} - \frac{1}{2N} \sqnorm{x_\star - x_0}
\end{align}
which implies $\sqnorm{\opF (x_{N-1})} \le \frac{4\sqnorm{x_0 - x_\star}}{\alpha^2 N^2}$.

On the other hand, in the stochastic setting, we do not have $\cQ_{N,j} \ge 0$ due to noise, and these quantities are not local as in typical stochastic analyses, making the convergence proof seemingly much harder.
However, we can isolate the dependence on stochastic errors at each iteration and control them through the following two lemmas.
These key results allow us to follow the reasoning similar to the deterministic case above.

\begin{lemma}
\label{lemma:extract-dependence-on-lower-bound}
Under the setting of \cref{theorem:stochastic-Dual-OHM-cocoercive}, let $e_k = \opF_{\cB_k}(x_k) - \opF(x_k)$ for $k=0,\dots,N-2$. Then
{\small
\begin{align}
\label{eqn:expectation-lower-bound-second}
\begin{aligned}
    0 & \ge \expec{\frac{(N-1)\alpha^2}{4} \sqnorm{\opF(x_{N-1})} + \frac{\alpha}{2} \inprod{\opF (x_{N-1})}{x_{N-1} - x_0}} -  \sum_{j=1}^{N-1} \lambda_{N,j} \left( \frac{\alpha\sigma^2}{2B} + \expec{\inprod{e_{j-1}}{\opT x_{N-1}}} \right) .
\end{aligned}
\end{align}
}
\end{lemma}

\begin{lemma}
\label{lemma:S-Dual-OHM-cross-term-expectation-bound}
Under the setting of \cref{theorem:stochastic-Dual-OHM-cocoercive}, $\expec{\inprod{e_{j-1}}{\opT x_{N-1}}} \le \frac{\alpha\sigma^2}{B}$ for $j=1,\dots,N-1$.
\end{lemma}

\cref{lemma:S-Dual-OHM-cross-term-expectation-bound} states that the amount by which the error $e_{j-1}$ propagates to the last iterate and affects the analysis is uniformly bounded by a constant independent of $N$. 
Combining Lemmas~\ref{lemma:extract-dependence-on-lower-bound} and \ref{lemma:S-Dual-OHM-cross-term-expectation-bound}, and using $\sum_{j=1}^{N-1} \lambda_{N,j} = \frac{\alpha(N-1)}{2}$, we can proceed as in \eqref{eqn:deterministic-case-final-bound} to obtain
\begin{align*}
    0 & \ge \expec{\frac{N\alpha^2}{4} \sqnorm{\opF(x_{N-1})} + \frac{\alpha}{2} \inprod{\opF (x_{N-1})}{x_\star - x_0}} - \frac{3(N-1)\alpha^2\sigma^2}{4B} \\
    & \ge \expec{\frac{N\alpha^2}{4} \sqnorm{\opF(x_{N-1})} - \frac{N\alpha^2}{8} \sqnorm{\opF (x_{N-1})} - \frac{1}{2N} \sqnorm{x_0 - x_\star}} - \frac{3(N-1)\alpha^2\sigma^2}{4B} .
\end{align*}
Rearranging, we have the desired bound of \cref{theorem:stochastic-Dual-OHM-cocoercive}.

\subsubsection{Intuition for why S-Dual-OHM is stable under noise}
\label{section:intuitive-explanation}

The previous analysis shows that the net accumulation of error in \ref{eqn:stochastic-Dual-OHM} is proportional to $\sum_{j=1}^{N-1} \lambda_{N,j} = \frac{\alpha(N-1)}{2} = \cO(N)$.
On the other hand, stochastic OHM satisfies the similar identity
\begin{align*}
    0 & = \frac{(N-1)\alpha^2}{4} \sqnorm{\opF (x_{N-1})} + \frac{\alpha}{2} \inprod{\opF(x_{N-1})}{x_{N-1} - x_0} + \sum_{j=1}^{N-1} \nu_{j+1,j} \cQ_{j+1,j}
\end{align*}
where $\cQ_{j+1,j} = \inprod{x_{j-1} - x_{j}}{\opF_{\cB_{j-1}}(x_{j-1}) - \opF_{\cB_{j}}(x_{j})} - \frac{\alpha}{2} \sqnorm{\opF_{\cB_{j-1}} (x_{j-1}) - \opF_{\cB_{j}} (x_{j})}$ but with $\nu_{j+1,j} = \frac{\alpha j(j+1)}{2N}$.
Here each $\cQ_{j+1,j}$ is not nonnegative, but contributes a negative lower bound proportional to $\sigma^2$ in expectation, which then gets multiplied with $\sum_{j=1}^{N-1} \nu_{j+1,j} = \Theta(N^2)$ and the resulting error term is no longer controllable as in \cref{theorem:stochastic-Dual-OHM-cocoercive}.
This is a central difference that makes the dual-anchor mechanism more amenable to stochastic extension than anchoring.
A concurrent work \cite{yoonTheoryCompositionDuality} has also provided a related theory that the sum of weights associated with inequalities used in the convergence proof measures a fixed-point algorithm's robustness to oracle noise, albeit in deterministic setting.

A complementary interpretation is that OHM is an anytime optimal algorithm, achieving the rate $\sqnorm{\opF(x_{k-1})} \le \frac{4 \sqnorm{x_0 - x_\star}}{\alpha^2 k^2}$ for all $k=1,2,\dots$. 
This means that OHM must maintain the optimal guarantee at every iteration, forcing it to exploit the geometric property of $\opF$, i.e.\ cocoercivity, throughout the trajectory. 
Its analysis is therefore more tightly tuned to the exact problem class and is more easily disrupted by oracle noise. 
\ref{eqn:Dual-OHM}, by contrast, is optimized only for the prescribed last iterate and does not suffer from the same accumulation of errors.

\subsection{Faster convergence under strong monotonicity}
\label{section:contractive-dual-ohm}

In this section, we show that if $\opF$ is additionally strongly monotone, then \ref{eqn:stochastic-Dual-OHM} can achieve the desired accuracy $\expec{\norm{\opF (x_{k})}} \le \epsilon$ much earlier than the prescribed horizon $N$.
Suppose that $\opF$ is $\mu$-strongly monotone and $\nicefrac{1}{L}$-cocoercive.
Then, with $0<\alpha<\frac{2}{L}$, $\opT = \opI - \alpha\opF$ is $\gamma$-\textit{contractive} (i.e., $\gamma$-Lipschitz) with $\gamma = \sqrt{1 - \alpha\mu (2-\alpha L)}$ because
\begin{align*}
    \sqnorm{\opT(x) - \opT(y)} & = \sqnorm{x-y} - 2\alpha \inprod{x - y}{\opF(x) - \opF(y)} + \alpha^2 \sqnorm{\opF(x) - \opF(y)} \\
    & \le \sqnorm{x-y} - \alpha(2-\alpha L) \inprod{x-y}{\opF(x) - \opF(y)} \le \left(1 - \alpha\mu (2-\alpha L) \right) \sqnorm{x-y} .
\end{align*}

\begin{theorem}
\label{theorem:stochastic-dual-ohm-strongly-monotone}
Consider the root-finding problem~\eqref{eqn:root-finding-problem} with $\nicefrac{1}{L}$-cocoercive and $\mu$-strongly monotone $\opF$ and a solution $x_\star$, satisfying Assumption~\ref{assumption:stochastic-operator-oracle}.
Let $0 < \alpha < \frac{2}{L}$ and $\gamma = \sqrt{1-\alpha\mu(2-\alpha L)} \in [0,1)$.
Let $y_k$ be the iterates from deterministic \ref{eqn:Dual-OHM} using true evaluations of $\opT$, with $x_0 = y_0$.
Then, for \(k=0,\dots,N-1\), \ref{eqn:stochastic-Dual-OHM} with $\opT_{\cB_k} = \opI - \alpha\opF_{\cB_k}$ satisfies
\begin{equation}
\label{eqn:tail-difference-bound}
    \expec{\sqnorm{x_k - y_{N-1}}} \le 2\left(\frac{1+\gamma}{1-\gamma}\right)^2 \gamma^{2k} \sqnorm{x_0 - x_\star} + \frac{2\alpha^2 \sigma^2}{B(1-\gamma^2)} .
\end{equation}
In particular, when $\alpha = \frac{1}{L}$, we have
\begin{equation}
\label{eqn:stochastic-dual-ohm-strongly-monotone-rate}
\begin{aligned}
    \expec{\norm{\opF(x_k)}^2} & \le \frac{8L^2 \sqnorm{x_0 - x_\star}}{N^2} + \frac{64L^4}{\mu^2} e^{-k\frac{\mu}{L}} \sqnorm{x_0 - x_\star} + \frac{L}{\mu} \frac{4\sigma^2}{B}.
\end{aligned}
\end{equation}
\end{theorem}

Notably, \cref{theorem:stochastic-dual-ohm-strongly-monotone} uses only \cref{assumption:stochastic-operator-oracle} and cocoercivity of $\opF$, but not \cref{assumption:cocoercivity-in-expectation}.

The reason why this result holds is two-fold. 
First, when $\opT$ is contractive, \ref{eqn:stochastic-Dual-OHM} does not deviate far from deterministic \ref{eqn:Dual-OHM} (\cref{lemma:stochastic-trajectory-deviation-bound}).
Second, deterministic \ref{eqn:Dual-OHM} achieves the designated terminal accuracy quickly as it linearly converges to its final iterate $y_{N-1}$ (\cref{lemma:deterministic-tail-difference-bound}).
Because \ref{eqn:Dual-OHM} has a small residual norm guarantee at $y_{N-1}$, this implies that $x_k$ also attains the comparably small residual quickly.

\begin{lemma}
\label{lemma:stochastic-trajectory-deviation-bound}
Under the conditions of \cref{theorem:stochastic-dual-ohm-strongly-monotone}, 
$\expec{\sqnorm{x_k - y_k}} \le \frac{\alpha^2 \sigma^2}{B(1 - \gamma^2)}$ for $k=0,\dots,N-1$.
\end{lemma}

\begin{lemma}
\label{lemma:deterministic-tail-difference-bound}
Under the conditions of \cref{lemma:stochastic-trajectory-deviation-bound}, we have $\norm{y_k-y_{N-1}} \le \left(\frac{1+\gamma}{1-\gamma}\right) \gamma^{k} \norm{x_0 - x_\star}$ for $k=0,\dots,N-1$ and $x_\star \in \Fix \opT$.
\end{lemma}

\begin{proof}[Proof of \cref{theorem:stochastic-dual-ohm-strongly-monotone}]
Note that $\expec{\norm{x_k - y_{N-1}}^2} \le 2\expec{\norm{x_k - y_k}^2} + 2\expec{\sqnorm{y_k - y_{N-1}}}$.
Then, by Lemmas~\ref{lemma:stochastic-trajectory-deviation-bound} and \ref{lemma:deterministic-tail-difference-bound}, $\expec{\norm{x_k - y_{N-1}}^2} \le \frac{2\alpha^2 \sigma^2}{B(1-\gamma^2)} + 2\left(\frac{1+\gamma}{1-\gamma}\right)^2 \gamma^{2k} \sqnorm{x_0 - x_\star}$, which proves \eqref{eqn:tail-difference-bound}.
With $\alpha = \frac{1}{L}$, we have $1-\gamma^2 = \frac{\mu}{L}$, so combining \eqref{eqn:tail-difference-bound} with \cref{proposition:deterministic-acceleration} we obtain
\begin{align*}
    \expec{\norm{\opF(x_k)}^2} & \le 2\expec{\sqnorm{\opF(y_{N-1})}} + 2\expec{\sqnorm{\opF(x_k) - \opF(y_{N-1})}} \\
    & \le 2\expec{\sqnorm{\opF(y_{N-1})}} + 2L^2 \expec{\sqnorm{x_k - y_{N-1}}} \\
    & \le \frac{8L^2 \sqnorm{x_0 - x_\star}}{N^2} + 2L^2 \left[ \frac{2\sigma^2}{\mu L B} + 2\left(\frac{1+\gamma}{1-\gamma}\right)^2 \gamma^{2k} \sqnorm{x_0 - x_\star} \right] .
\end{align*}
where the second inequality uses $L$-Lipschitzness of $\opF$.
Rearranging and using $\gamma < 1$, $(1-\gamma)^2 \ge \frac{(1-\gamma)^2 (1+\gamma)^2}{4} = \frac{\mu^2}{4L^2}$ and $\gamma^{2k} = \left(1 - \frac{\mu}{L}\right)^k \le e^{-k\frac{\mu}{L}}$, we obtain \eqref{eqn:stochastic-dual-ohm-strongly-monotone-rate}.
\end{proof}

Let $\norm{x_0 - x_\star} = D$. The bound \eqref{eqn:stochastic-dual-ohm-strongly-monotone-rate} shows that if one chooses, e.g., $N \ge \frac{4LD}{\epsilon}, k \ge \frac{L}{\mu} \log \frac{256 L^4 D^2}{\mu^2\epsilon^2}$ and $B \ge \frac{16L\sigma^2}{\mu \epsilon^2}$, then $\expec{\norm{\opF(x_k)}}^2 \le \expec{\norm{\opF(x_k)}^2} \le \epsilon^2$.
If we run \ref{eqn:stochastic-Dual-OHM} with these choices of $N$ and $B$, we will achieve the accuracy $\epsilon$ in $\cO\left(\frac{L}{\mu} \log \epsilon^{-1}\right)$ iterations without the need to run all $N$ iterations.
With this early stopping, \ref{eqn:stochastic-Dual-OHM} has improved $\cO\left(\frac{L^2\sigma^2}{\mu^2 \epsilon^2} \log \frac{1}{\epsilon} \right)$ oracle complexity for strongly monotone $\opF$, 
which has near-optimal dependence on $\epsilon$ \cite{chenNearoptimalAlgorithmsMaking2024}.
The dependence on condition number, however, is still suboptimal, and improving it based on the \ref{eqn:Dual-OHM} mechanism remains an open question.

\section{Numerical Experiments}
\label{section:experiments}

We provide numerical evaluations to demonstrate that the dual-anchor mechanism is indeed preferable to anchoring (Halpern algorithms) under stochasticity, and compare the empirical performance of \ref{eqn:stochastic-Dual-OHM} with prior algorithms for stochastic monotone inclusions/fixed-point problems.
We consider three distinct problems: the worst-case nonexpansive fixed-point operator designed by \citep{ParkRyu2022_exact}, a finite-sum cocoercive operator using uniform random vectors from the unit sphere, and a strongly-convex--strongly-concave minimax problem with Huber-type regularizers following \citep{yoonAcceleratedAlgorithmsSmooth2021, chenNearoptimalAlgorithmsMaking2024}.
All experiments were run on a MacBook Air with Apple M3 Chip and 24GB Memory.
 
\paragraph{Baseline algorithms.}
We consider the simple algorithm $x_{k+1} = x_k - \alpha \opF_{\cB_k}(x_k)$, which we call Stochastic Gradient Descent-Ascent (\algname{SGDA}) following the nomenclature of the minimax optimization literature.
We consider the Stochastic Extragradient (\algname{SEG}) with the update rule
\[
    x_{k+1/2} = x_k - \alpha \opF_{\cB_k}(x_k) , \quad x_{k+1} = x_k - \alpha \opF_{\cB_{k+1/2}} (x_{k+1/2}) .
\]
Next, we consider the naive stochastic extension of \ref{eqn:OHM}: $x_{k+1} = \frac{1}{k+2}x_0 + \frac{k+1}{k+2} \left( x_k-\alpha \opF_{\cB_k}(x_k)\right)$, using the same constant mini-batch size as \ref{eqn:stochastic-Dual-OHM}. We refer to this algorithm as \algname{S-OHM} (with constant $B$). This is included to isolate the effect of using the dual-anchor update rather than the usual anchor acceleration and directly contrast the two mechanisms.

We also include the algorithms carefully designed for reducing the expected last-iterate residual.
First, the \algname{Halpern-PAGE} algorithm from \citep{cai2022stochastic} uses the \algname{S-OHM} update but replaces $\opF_{\cB_k}(x_k)$ by a PAGE-type estimator which uses a larger batch size $s_1$ with probability $p_k = \min\left\{1, \frac{2}{k+2}\right\}$ and otherwise uses a sample operator difference with batch size $s_2$.
Second, we use the single-loop simplification of \algname{RAIN} algorithm from \citep{chenNearoptimalAlgorithmsMaking2024} used in their experiments, which can be written as 
\[
    z_{k+1/2} = z_k - \eta\bigl(\widehat{\opF}(z_k; \xi_k) + r_k(z_k)\bigr),
    \qquad
    z_{k+1} = z_k - \eta\bigl(\widehat{\opF}(z_{k+1/2}; \xi_{k+1/2}) + r_k(z_{k+1/2})\bigr)
\]
where $r_k(x) = \lambda\gamma (S_k x - m_k) , S_k = \sum_{t<k}(1+\gamma)^t, m_k = \sum_{t<k}(1+\gamma)^t z_t$, and $\eta, \lambda, \gamma > 0$ are hyperparameters.
Finally, \algname{Halpern-VR-Finite} from \citep{caiVarianceReducedHalpern2024} is similar to \algname{Halpern-PAGE} but applicable only to finite sum problems of the form $\opF = \frac{1}{n} \sum_{i=1}^n \opF(\cdot; \xi_i)$ and uses either full-batch (true) operator $\opF$ or batches of size $\lceil\sqrt{n}\rceil$.

We tune all baseline algorithms and \ref{eqn:stochastic-Dual-OHM} under the same total sample budget $Q=NB$ via grid search, where computing $\opF_{\cB_k}$ with $|\cB_k| = B$ counts as $B$ operator evaluations.
The only hyperparameter $B$ for \ref{eqn:stochastic-Dual-OHM} is tuned over $B\in\{1,10,20,50,100\}$, and \algname{S-OHM} (Constant $B$) is then run with the same selected $B$.\footnote{While this emphasizes the distinction between the anchor and the dual-anchor updates, for fully controlled comparison, we additionally provide experiments where $B$ for \algname{S-OHM} is tuned separately in Appendix~\ref{section:additional-experiments}.}
The \algname{Halpern-PAGE} parameters are tuned over $s_1\in\{1,10,20,50,100\}$ and $s_2\in\{1,5,10,20\}$, and \algname{RAIN} hyperparameters are tuned over $\eta\in\{0.005,0.01,0.05,0.1,1\}$, $\lambda\in\{0.001,0.01,0.1,1\}$, and $\gamma \in \{0.001,0.01,0.1,1\}$.
\algname{SGDA} and \algname{SEG} step-sizes were tuned over $\alpha \in \{0.005,0.01,0.05,0.1,1\}$, the same step-size grid as \algname{RAIN}.
We run each experiment with 10 independent random seeds and report the average residual norm $\norm{\opF(x_k)}$ together with the shaded region indicating the empirical 5th--95th percentile band.

\begin{figure}[t]
    \centering
    \includegraphics[width=0.47\linewidth]{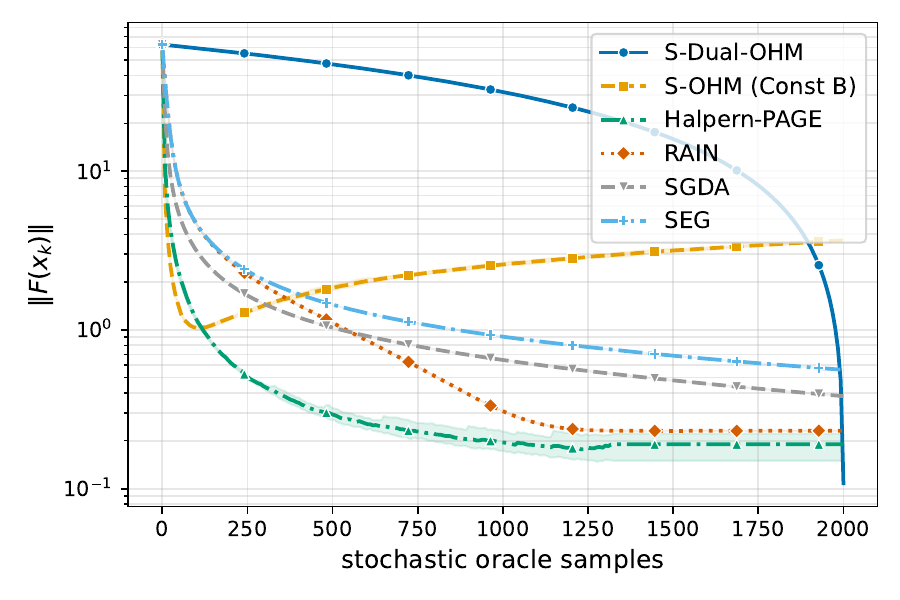}
    \hfill
    \includegraphics[width=0.47\linewidth]{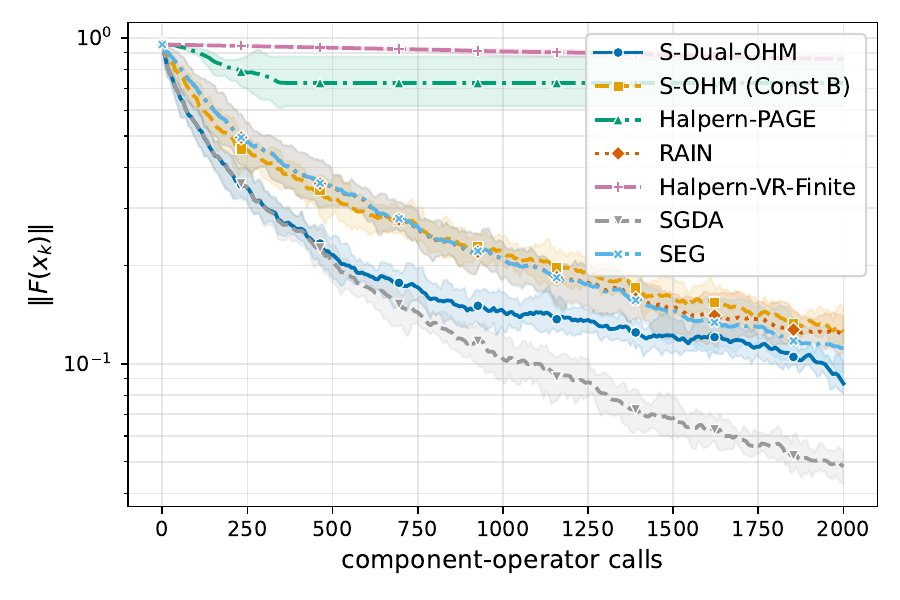}
    \caption{Plot of residual norm versus stochastic samples. 
    \textbf{(Left)} Worst-case nonexpansive operator construction (Experiment 1). 
    \textbf{(Right)} Finite-sum random cocoercive operator construction (Experiment 2).
    Solid curves indicate means over 10 independent runs, and shaded bands denote empirical 5th--95th
    percentiles.}
    \label{figure:numerical-experiments-1}
\end{figure}

\paragraph{Experiment 1: Worst-case nonexpansive operator.}
We use the affine lower-bound construction of \citep{ParkRyu2022_exact}: define $\opH\colon \reals^d \to \reals^d$ by $\opH(x_1,\dots,x_d) = \left( x_d - \frac{2}{\sqrt{d}} , -x_1, -x_2, \dots, -x_{d-1} \right)$, and after choosing a shift vector $s\in \reals^d$, let $\opF(x) = \opH(x-s) + (x-s)$ and $\opT = \opI - \opF$.
Then $\opF$ is $\frac{1}{2}$-cocoercive and $\opT$ is nonexpansive.
We use Gaussian noise oracle $\widehat\opF (x;\xi) = \opF(x) + \xi$ with $\xi \sim \cN\left(0,\frac{\sigma^2}{d}I_d\right)$, and use a budget of $2000$ samples. 
We set $d=2001$, $x_0 = 0$, $\alpha=1$, $\sigma=0.1$ and $s \sim \cN(0,I_d)$.

The left panel of Figure~\ref{figure:numerical-experiments-1} illustrates that \ref{eqn:stochastic-Dual-OHM} is not an anytime algorithm, so its intermediate residuals can be worse than those of the other algorithms; however, in the end, it achieves the smallest last-iterate residual among all algorithms compared.
This is in sharp contrast with \algname{S-OHM} with constant $B$, which fails to converge and exhibits diverging residual.
This demonstrates that Halpern-type algorithms can indeed be prone to accumulation of stochastic errors.
However, with variance reduction (\algname{Halpern-PAGE}) or recursively adjusted anchor (\algname{RAIN}), the residual convergence gets stabilized.
Finally, \algname{SEG} and \algname{SGDA} converge steadily although they were not specifically designed for residual minimization.

\paragraph{Experiment 2: Finite-sum cocoercive operator.}
We consider a finite-sum problem
\[
    \opF(x)=\frac1n\sum_{i=1}^n \opF_i(x),
    \qquad
    \opF_i(x)=
    \begin{pmatrix}
        q_i(q_i^\top x_{1:r})\\
        0
    \end{pmatrix}\in\mathbb R^d,
\]
where each $q_i\in\mathbb R^r$ is sampled uniformly from the unit sphere.
Here each $\opF_i$ is $1$-cocoercive and linear, and therefore $\opF$ is cocoercive with $L\le 1$ but not strongly monotone when $r<d$ since the common null space of $\opF_i$ is nontrivial.
A stochastic oracle call samples component indices and returns the average of the queried $\opF_i$'s.
The operator variance is bounded over the region that trajectories of all algorithms stay within during all experiment runs.
We use a budget of 2000 stochastic samples, and set $n=200$, $d=200$, $r=199$ and $\opT = \opI - \alpha\opF$ with $\alpha=1$.
We take $x_0 \sim 10 \cdot \cN(0, I_d)$.
This is a generic, non-worst-case problem that basic methods like SGDA can solve effectively.
Still, \ref{eqn:stochastic-Dual-OHM} is also fairly competent, and importantly, outperforms its primal counterpart \algname{S-OHM}. 
Notably, \algname{Halpern-VR-Finite} makes very small progress because the total sample budget 2000 is restrictive with our choice $n=200$, even though it has the best asymptotic dependence on $\epsilon$ in theory. 

\begin{figure}[t]
    \centering
    \includegraphics[width=0.47\linewidth]{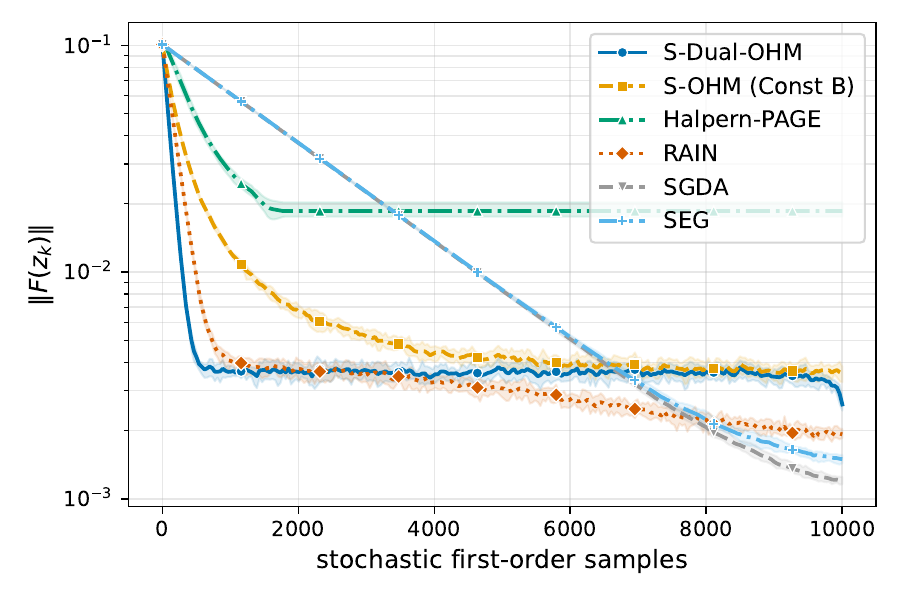}
    \hfill
    \includegraphics[width=0.47\linewidth]{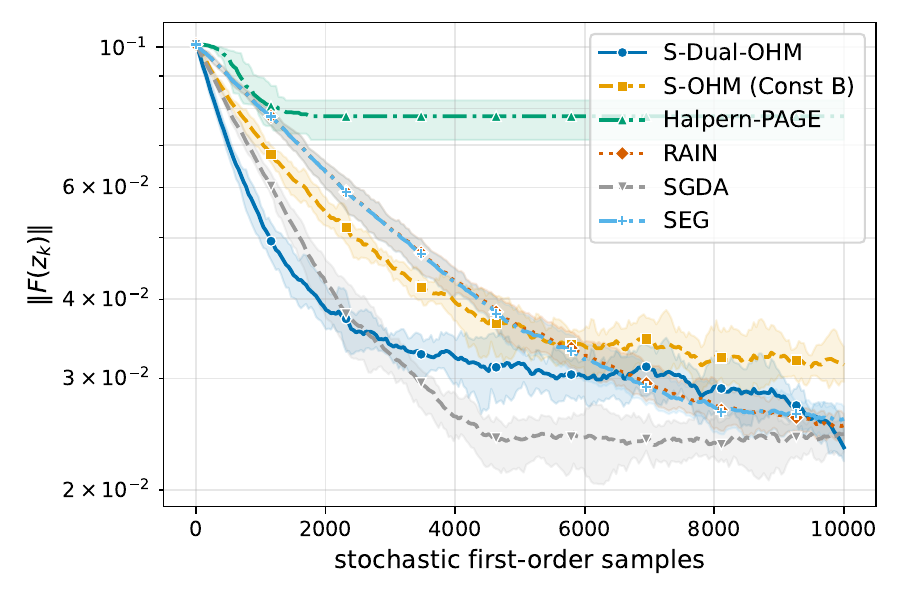}
    \caption{Plot of residual norm versus stochastic samples for SCSC Huber-type minimax problem (Experiment 3). 
    \textbf{(Left)} Low-variance regime $\sigma=0.05$. 
    \textbf{(Right)} High-variance regime $\sigma=1.5$.
    Solid curves are means over 10 independent runs, and shaded bands indicate empirical 5th--95th
    percentiles.}
    \label{figure:numerical-experiments-2}
\end{figure}

\paragraph{Experiment 3: SCSC minimax problem with Huber-type regularizers.}
We consider 
\[
    \Phi(x,y)
    =(1-\delta)\sum_i g_\nu(x_i)
     +\delta x^\top y
     -(1-\delta)\sum_i g_\nu(y_i)
     +\frac{\mu}{2}\norm{x}^2-\frac{\mu}{2}\norm{y}^2,
\]
for $z = (x,y) \in \reals^d \times \reals^d$, where $g_\nu(u)=
\begin{cases}
    \frac{1}{2} u^2, & |u|\le \nu \\
    \nu |u| - \frac{1}{2}\nu^2, & |u|>\nu 
\end{cases}$.
The saddle operator $\opF(z)=(\nabla_x\Phi(x,y),-\nabla_y\Phi(x,y))$ is then given as $\opF(z)=
\begin{pmatrix}
    (1-\delta)\operatorname{clip}(x,-\nu,\nu)+\delta y+\mu x\\
    -\delta x+(1-\delta)\operatorname{clip}(y,-\nu,\nu)+\mu y
\end{pmatrix}$,
which is $\mu$-strongly monotone and $M$-Lipschitz with $M=1+\mu$.  
We use the Gaussian noise oracle $\widehat\opF (z;\xi) = \opF(z) + \xi$ with $\xi \sim \cN\left(0,\frac{\sigma^2}{2d}I_{2d}\right)$, and use a budget of 10000 stochastic samples.
We set $d=50$, $\delta=10^{-2}$, $\nu=5 \times 10^{-5}$, $\mu=0.1$, $L=\frac{M^2}{\mu}$ and $z_0 = (x_0, y_0)$ is a uniform random unit vector.
We use $\alpha = \frac{1}{L}$ for \ref{eqn:stochastic-Dual-OHM}, \algname{S-OHM} and \algname{Halpern-PAGE}.

We simulate both low noise ($\sigma=0.05$) and high noise ($\sigma=1.5$) regimes.
When $\sigma=0.05$, we observe rapid progression in earlier iterations of \ref{eqn:stochastic-Dual-OHM}, but then it stagnates (Figure~\ref{figure:numerical-experiments-2}).
This is because the noise term is small, so the best batch size choice is $B^* = 1$, i.e., mini-batching does not benefit the final convergence.
In this regime, \algname{RAIN, SGDA, SEG} converges more steadily than \ref{eqn:stochastic-Dual-OHM}, although we still observe that \algname{S-OHM} is inferior to \ref{eqn:stochastic-Dual-OHM}.
With larger noise with $\sigma=1.5$, we observe that mini-batching becomes effective and $B^* = 10$ is selected; in this case, dual-anchor with mini-batching exhibits comparable performance to classical methods and deliberately designed near-optimal algorithm \algname{RAIN}.

\section{Conclusion and Open Questions}

We demonstrate that dual-anchor acceleration from deterministic root-finding or fixed-point problems admits a direct stochastic extension, unlike anchor-type accelerated algorithms. 
This is based on the observation that distinct deterministic accelerations are intrinsically different in their behavior under stochasticity. 
We believe that our work enables interesting future work on acceleration and optimal complexities for stochastic optimization.

Two immediate limitations of our work, which thus remain as open questions, are: \textbf{(1)} whether the dual-anchor mechanism extends to root-finding problems with monotone and Lipschitz operators $\opF$ and \textbf{(2)} whether we can further improve the complexity to the optimal level \citep{chenNearoptimalAlgorithmsMaking2024} $\cO(\epsilon^{-2} \log\epsilon^{-1})$ via dual-anchor acceleration.
As the dual-anchor acceleration also exists in the monotone Lipschitz inclusion setting \citep{yoonOptimalAccelerationMinimax2024}, we believe that the extension would be possible with novel insights not relying on nonexpansivity of $\opI - \alpha\opF$.
We also expect our result to serve as a baseline that can be combined with broader techniques in stochastic optimization for potentially achieving the optimal complexity.

\section*{Acknowledgments}
TaeHo Yoon’s contribution to this work was supported by NSF CCF 2504626.
Nicolas Loizou’s contribution to this work was supported by NSF CCF 2504626 and NSF CAREER 2542902.

\bibliographystyle{plainnat}
\bibliography{ref}

 \clearpage

\appendix

\part*{Appendix}

\medskip

\tableofcontents
\newpage

\section{Missing Proofs}
\label{section:appendix-missing-proofs}

We denote by $\cF_1 \subseteq \cF_2 \subseteq \dots$ the natural filtration generated by each iteration of the algorithm (selection of batches $\cB_i$), i.e., $\cF_k = \sigma\left( \cB_0, \cB_1, \dots, \cB_{k-1} \right)$.

\subsection{Proof of \texorpdfstring{\cref{lemma:stochastic-Dual-OHM-magic-identity}}{Lemma 4.2}}

\begin{proof}
The proof is purely algebraic, and does not use any property of $\opF$ as an operator or stochastic property of the mini-batches $\cB_j$.
The case \(N=1\) is immediate (we have the vacuous identity $0=0$), so assume \(N\ge 2\). 

For \(j=1,\dots,N\), define
\[
    G_j = \opF_{\cB_{j-1}}(x_{j-1}),
    \qquad
    g_j = \frac{\alpha}{2}G_j,
    \qquad
    u_j = x_{j-1}-g_j .
\]
Here, $\cB_{N-1}$ is a ghost (auxiliary) batch introduced only for the analysis, and is not used by the algorithm update; in fact, one may ignore this ghost batch (i.e., safely assume that $\cB_{N-1}$ is the full batch) and simply let $G_N = \opF(x_{N-1})$.
With this notation, we have
\[
    \opT_{\cB_{j-1}}(x_{j-1})
    = x_{j-1}-\alpha G_j
    = x_{j-1}-2g_j .
\]
Because \(\opT_{\cB_{-1}}(x_{-1})=x_0\), the \(k=0\) update of \ref{eqn:stochastic-Dual-OHM} is
\[
    x_1=x_0+\frac{N-1}{N}\bigl(\opT_{\cB_0}(x_0)-x_0\bigr).
\]
Let
\[
    a_k=\frac{N-k-1}{N-k},
    \qquad
    z_0=0,
    \qquad
    z_{k+1}=\opT_{\cB_k}(x_k)-x_{k+1}
\]
for \(k=0,\dots,N-2\). Then, from the update rule,
\begin{align}
    z_{k+1}
    & = \opT_{\cB_k}(x_k)
       -x_k-a_k\bigl(\opT_{\cB_k}(x_k)-\opT_{\cB_{k-1}}(x_{k-1})\bigr) \nonumber \\
    & = (1-a_k) (x_k - 2g_{k+1}) - x_k + a_k \opT_{\cB_{k-1}}(x_{k-1}) \nonumber \\ 
    & = a_k \left( \opT_{\cB_{k-1}}(x_{k-1}) - x_k \right) - 2(1-a_k) g_{k+1} \nonumber \\
    & = \frac{N-k-1}{N-k} z_k - \frac{2}{N-k} g_{k+1} . \label{eqn:zk-recursion}
\end{align}
Then, by definition of $z_{k+1}$,
\begin{align}
\label{eqn:xk-recursion-with-z}
    x_{k+1} = \opT_{\cB_k}(x_k)-z_{k+1} = x_k-2g_{k+1}-z_{k+1}.
\end{align}

Now define, for \(k=0,\dots,N-1\),
\[
    V_k = -\frac{N-k-1}{N-k}\sqnorm{z_k+2g_N}
    +\frac{2}{N-k}\inprod{z_k+2g_N}{x_k-x_{N-1}} .
\]
Note that \(V_{N-1}=0\).
Next, we analyze \(V_k-V_{k+1}\). To simplify the notations, fix \(k\in\{0,\dots,N-2\}\), and write
\[
    m=N-k,\qquad
    w_k=z_k+2g_N,\qquad
    d_k=g_{k+1}-g_N,\qquad
    p_k=x_k-x_{N-1}.
\]
From \eqref{eqn:zk-recursion}, we have
\[
    w_{k+1} - 2g_N = \frac{N-k-1}{N-k} (w_k - 2g_N) - \frac{2}{N-k} g_{k+1} \implies w_{k+1} = \frac{m-1}{m}w_k-\frac{2}{m}d_k
\]
and \eqref{eqn:xk-recursion-with-z} yields
\[
    p_{k+1} = x_{k+1} - x_{N-1} = (x_k - x_{N-1}) - (z_{k+1} + 2g_N) - 2(g_{k+1} - g_N) = p_k -w_{k+1} - 2d_k .
\]
Plugging the last identity into the expansion of $V_k - V_{k+1}$ below, we proceed as
\[
\begin{aligned}
V_k-V_{k+1}
&=
-\frac{m-1}{m}\sqnorm{w_k}
+\frac{2}{m}\inprod{w_k}{p_k}
+\frac{m-2}{m-1}\sqnorm{w_{k+1}}
-\frac{2}{m-1}\inprod{w_{k+1}}{p_{k+1}}                                      \\
&=
-\frac{m-1}{m}\sqnorm{w_k}
+\frac{m}{m-1}\sqnorm{w_{k+1}}
+\inprod{\frac{2}{m}w_k-\frac{2}{m-1}w_{k+1}}{p_k}
+\frac{4}{m-1}\inprod{w_{k+1}}{d_k}                                           \\
&=
\frac{4}{m(m-1)}\inprod{p_k}{d_k}
-\frac{4}{m(m-1)}\sqnorm{d_k}                                                   \\
&=
\frac{4}{(N-k)(N-k-1)}
\inprod{u_{k+1}-u_N}{g_{k+1}-g_N},
\end{aligned}
\]
where in the third equality we substituted
\(w_{k+1}=\frac{m-1}{m}w_k-\frac{2}{m}d_k\) and canceled out the terms, and in the last equality we used
\[
    p_k-d_k
    =
    x_k-x_{N-1}-(g_{k+1}-g_N)
    =
    u_{k+1}-u_N .
\]
Summing over \(k=0,\dots,N-2\) and using \(V_{N-1}=0\) yields
\[
    V_0
    =
    \sum_{j=1}^{N-1}
    \frac{4}{(N-j+1)(N-j)}
    \inprod{u_j-u_N}{g_j-g_N}.
\]
On the other hand, since \(z_0=0\),
\[
    V_0
    =
    -\frac{N-1}{N}\sqnorm{2g_N}
    +\frac{2}{N}\inprod{2g_N}{x_0-x_{N-1}}
    =
    -\frac{4}{N}
    \left(
        (N-1)\sqnorm{g_N}
        +\inprod{g_N}{x_{N-1}-x_0}
    \right).
\]
Therefore, equating the two expressions for $V_0$ and multiplying $\frac{N}{4}$ throughout, we obtain
\begin{align}
\label{eqn:magic-identity-with-substitution}
    0
    =
    (N-1)\sqnorm{g_N}
    +\inprod{g_N}{x_{N-1}-x_0}
    +
    \sum_{j=1}^{N-1}
    \frac{N}{(N-j)(N-j+1)}
    \inprod{u_j-u_N}{g_j-g_N}.
\end{align}
Finally, substituting \(u_j=x_{j-1}-g_j\) and \(g_j=\frac{\alpha}{2}G_j\), we have
\[
\begin{aligned}
    \inprod{u_j-u_N}{g_j-g_N}
    &=
    \inprod{x_{j-1}-x_{N-1}}{g_j-g_N}
    -\sqnorm{g_j-g_N} \\
    &=
    \frac{\alpha}{2}
    \left(
        \inprod{x_{j-1}-x_{N-1}}{G_j-G_N}
        -\frac{\alpha}{2}\sqnorm{G_j-G_N}
    \right),
\end{aligned}
\]
and
\[
    (N-1)\sqnorm{g_N}
    =
    \frac{(N-1)\alpha^2}{4}\sqnorm{G_N},
    \qquad
    \inprod{g_N}{x_{N-1}-x_0}
    =
    \frac{\alpha}{2}\inprod{G_N}{x_{N-1}-x_0}.
\]
Hence, recalling \(G_j=\opF_{\cB_{j-1}}(x_{j-1})\) and $\lambda_{N,j} = \frac{N\alpha}{2(N-j)(N-j+1)}$, \eqref{eqn:magic-identity-with-substitution} is exactly the claimed identity.
\end{proof}

\subsection{Proof of \texorpdfstring{\cref{lemma:extract-dependence-on-lower-bound}}{Lemma 4.3}}

We can rewrite \cref{lemma:stochastic-Dual-OHM-magic-identity} as
\begin{align}
\label{eqn:expectation-lower-bound}
\begin{aligned}
    0 & \ge \expec{\frac{(N-1)\alpha^2}{4} \sqnorm{\opF(x_{N-1})} + \frac{\alpha}{2} \inprod{\opF (x_{N-1})}{x_{N-1} - x_0}} \\ 
    & \quad + \sum_{j=1}^{N-1} \lambda_{N,j} \expec{\inprod{x_{j-1} - x_{N-1}}{\opF_{\cB_{j-1}}(x_{j-1}) - \opF(x_{N-1})} - \frac{\alpha}{2} \sqnorm{\opF_{\cB_{j-1}} (x_{j-1}) - \opF (x_{N-1})}} .
\end{aligned}
\end{align}
Now let $e_{j-1} = \opF_{\cB_{j-1}}(x_{j-1}) - \opF(x_{j-1})$ be the error term at iteration $j-1$. 
Then we can write
\begin{align*}
    & \expec{\inprod{x_{j-1} - x_{N-1}}{\opF_{\cB_{j-1}}(x_{j-1}) - \opF(x_{N-1})} - \frac{\alpha}{2} \sqnorm{\opF_{\cB_{j-1}} (x_{j-1}) - \opF (x_{N-1})} } \\
    & = \expec{\inprod{x_{j-1} - x_{N-1}}{\opF(x_{j-1}) - \opF(x_{N-1})} - \frac{\alpha}{2} \sqnorm{\opF (x_{j-1}) - \opF (x_{N-1})} } \\
    & \quad + \expec{\inprod{x_{j-1} - x_{N-1}}{e_{j-1}} - \alpha \inprod{\opF(x_{j-1}) - \opF(x_{N-1})}{e_{j-1}} - \frac{\alpha}{2} \sqnorm{e_{j-1}}} \\
    & \ge \expec{\inprod{- x_{N-1} + \alpha \opF(x_{N-1})}{e_{j-1}}} - \frac{\alpha \sigma^2}{2B} 
\end{align*}
where the last inequality uses cocoercivity of $\opF$, $\condexp{e_{j-1}}{\cF_{j-1}} = 0$ and $x_{j-1} - \alpha \opF(x_{j-1})$ is deterministic on $\cF_{j-1}$, and $\condexp{\sqnorm{e_{j-1}}}{\cF_{j-1}} \le \frac{\sigma^2}{B}$.
Substituting the above into \eqref{eqn:expectation-lower-bound} and using $\opT = \opI - \alpha\opF$, we obtain~\eqref{eqn:expectation-lower-bound-second}.    

\subsection{Proof of \texorpdfstring{\cref{lemma:S-Dual-OHM-cross-term-expectation-bound}}{Lemma 4.4}}

\begin{lemma}
\label{lemma:convex-representation}
For $k=1,\dots,N-1$, the \ref{eqn:stochastic-Dual-OHM} iterates can be expressed as
\begin{equation}
\label{eq:convex-representation}
x_k = \frac{1}{N}x_0 + \sum_{t=0}^{k-2} \frac{1}{(N-t-1)(N-t)} \opT_{\cB_t}(x_t) + \frac{N-k}{N-k+1}\opT_{\cB_{k-1}}(x_{k-1}).
\end{equation}
\end{lemma}

\begin{proof}
We use induction on \(k\). For \(k=1\), $x_1 = x_0+\frac{N-1}{N}\bigl(\opT_{\cB_0}(x_0)-x_0\bigr) = \frac1N x_0 + \frac{N-1}{N}\opT_{\cB_0}(x_0)$, which agrees with \eqref{eq:convex-representation}.
Now assume \eqref{eq:convex-representation} holds for some \(k\in\{1,\dots,N-2\}\).
Then, by induction hypothesis,
\begin{align*}
x_{k+1} & = x_k+\frac{N-k-1}{N-k}\Bigl(\opT_{\cB_k}(x_k)-\opT_{\cB_{k-1}}(x_{k-1})\Bigr) \\
& = \frac1N x_0 + \sum_{t=0}^{k-2} \frac{1}{(N-t-1)(N-t)} \opT_{\cB_t}(x_t) + \left( \frac{N-k}{N-k+1}-\frac{N-k-1}{N-k} \right)\opT_{\cB_{k-1}}(x_{k-1}) \\
& \quad + \frac{N-k-1}{N-k}\opT_{\cB_k}(x_k).
\end{align*}
The induction (hence the proof) is then complete, as $\frac{N-k}{N-k+1}-\frac{N-k-1}{N-k} = \frac{1}{(N-k)(N-k+1)}$.
\end{proof}

\begin{lemma}[Leave-one-out stability]
\label{lemma:loo-stability}
For \(s\in\{0,\dots,N-2\}\), let \(\big\{x_k^{(s)}\big\}_{k=0}^{N-1}\) be the trajectory obtained from \ref{eqn:stochastic-Dual-OHM} by replacing only the minibatch \(\cB_s\) by an independent copy \(\cB_s'\), while keeping all other minibatches the same in the two runs.
Then, for every \(k=s+1,\dots,N-1\), 
\begin{equation}
\label{eq:loo-stability-general}
\expec{\sqnorm{x_k-x_k^{(s)}}}
\le
\frac{\alpha^2 \sigma^2}{B}
\left(
1+\frac{(N-k)(N-k-1)}{(N-s)(N-s-1)}
\right).
\end{equation}
In particular,
\begin{equation}
\label{eq:loo-stability-terminal}
\expec{\sqnorm{x_{N-1}-x_{N-1}^{(s)}}}
\le
\frac{\alpha^2 \sigma^2}{B}.
\end{equation}
\end{lemma}

\begin{proof}
Define $w_{k,t} = \begin{cases}
\dfrac{1}{(N-t-1)(N-t)}, & 0\le t\le k-2,\\[2mm]
\dfrac{N-k}{N-k+1}, & t=k-1
\end{cases}$ so that by \cref{lemma:convex-representation},
\[
    x_k=\frac1N x_0+\sum_{t=0}^{k-1} w_{k,t}\opT_{\cB_t}(x_t),
    \qquad
    x_k^{(s)}=\frac1N x_0+\sum_{t=0}^{k-1} w_{k,t}\opT_{\cB_t^{(s)}}(x_t^{(s)}),
\]
where \(\cB_t^{(s)}=\cB_t\) for \(t\neq s\), and \(\cB_s^{(s)}=\cB_s'\).
Since \(\sum_{t=0}^{k-1}w_{k,t}=1-\frac1N\), by Jensen's inequality,
\begin{align}
\label{eqn:dks-bound-Jensen}
\begin{aligned}
    d_k^{(s)} & := \expec{\sqnorm{x_k-x_k^{(s)}}} \le \expec{\sum_{t=0}^{k-1} w_{k,t} \sqnorm{ \opT_{\cB_t}(x_t)-\opT_{\cB_t^{(s)}} \big(x_t^{(s)}\big)}} \\
    & = \sum_{t=s}^{k-1} w_{k,t} \expec{\sqnorm{ \opT_{\cB_t}(x_t)-\opT_{\cB_t^{(s)}} \big(x_t^{(s)}\big)}} .
\end{aligned}
\end{align}
Let $\nu^2 = \frac{\alpha^2 \sigma^2}{B}$ for simplicity. We will show by induction on $k$ that 
\begin{align}
\label{eqn:dks-recursive-bound-formula}
    d_k^{(s)} \le \nu^2\left( 1+\frac{(N-k)(N-k-1)}{(N-s)(N-s-1)} \right) .
\end{align}
For $k=s+1$, from \eqref{eqn:dks-bound-Jensen} we have
\begin{align*}
    d_{s+1}^{(s)} \le w_{s+1,s} \expec{\sqnorm{ \opT_{\cB_s}(x_s) - \opT_{\cB_s'} \big(x_s \big) }} \le \frac{N-s-1}{N-s} \cdot 2\nu^2
\end{align*}
which agrees with \eqref{eqn:dks-recursive-bound-formula}. Here we used the fact that $\opT_{\cB_s}(x_s)$ and $\opT_{\cB_s'}(x_s)$ are estimators of $\opT(x_s)$ of variance $\le \frac{\alpha^2 \sigma^2}{B} = \nu^2$, and they are conditionally independent on $\cF_s$. 
Next, let $s+2 \le \ell \le N-1$, and suppose \eqref{eqn:dks-recursive-bound-formula} holds for all $t=s+1,\dots,\ell-1$.
Note that in \eqref{eqn:dks-bound-Jensen}, we have $\cB_t^{(s)} = \cB_t$ for $t>s$, so by \cref{assumption:cocoercivity-in-expectation},
\begin{align*}
    \expec{\sqnorm{ \opT_{\cB_t}(x_t)-\opT_{\cB_t^{(s)}} \big(x_t^{(s)}\big)}} & = \expec{\condexp{\sqnorm{ \opT_{\cB_t}(x_t) - \opT_{\cB_t} \big(x_t^{(s)}\big)}}{\cG_t^{(s)}}} \\
    & \le \expec{\condexp{\sqnorm{x_t - x_t^{(s)}}}{\cG_t^{(s)}}} = \expec{\sqnorm{x_t - x_t^{(s)}}} = d_t^{(s)} 
\end{align*}
where $\cG_t^{(s)} = \sigma(\cB_0, \dots, \cB_{t-1}, \cB_s')$ is the $\sigma$-algebra enlarged from $\cF_t$ to make $x_t^{(s)}$ measurable. 
Applying this to \eqref{eqn:dks-bound-Jensen} with the induction hypothesis, we obtain
\begin{equation}
\label{eq:loo-recursion}
\begin{aligned}
    d_\ell^{(s)} & \le 2 w_{\ell,s}\nu^2 + \sum_{t=s+1}^{\ell-1} w_{\ell,t} d_t^{(s)} \\
    & \le 2 w_{\ell,s}\nu^2 + \sum_{t=s+1}^{\ell-1} w_{\ell,t} \nu^2 \left( 1+\frac{(N-t)(N-t-1)}{(N-s)(N-s-1)} \right) \\
    & = \nu^2 \left[ 2 w_{\ell,s} + \sum_{t=s+1}^{\ell-1} w_{\ell,t} \right] + \frac{\nu^2}{(N-s)(N-s-1)} \sum_{t=s+1}^{\ell-1} w_{\ell,t}(N-t)(N-t-1) .
\end{aligned}
\end{equation}
Now using the explicit formula for \(w_{\ell,t}\), we can directly verify
\[
2 w_{\ell,s} + \sum_{t=s+1}^{\ell-1} w_{\ell,t} = 1 - \frac{1}{N-s} + \frac{1}{(N-s)(N-s-1)},
\]
and
\begin{align*}
    & \sum_{t=s+1}^{\ell-1} w_{\ell,t}(N-t)(N-t-1) = \sum_{t=s+1}^{\ell-2} 1 + \frac{N-\ell}{N-\ell+1} (N-\ell+1)(N-\ell) \\
    & = \ell-s-2 + (N-\ell)^2 = N-s-2 + (N-\ell)(N-\ell-1) .
\end{align*}
Hence
\begin{align*}
d_\ell^{(s)} &\le \nu^2 \left( 1-\frac{1}{N-s} \right) + \frac{\nu^2}{(N-s)(N-s-1)} \Bigl(N-s-1 + (N-\ell)(N-\ell-1)\Bigr) \\
& = \nu^2 \left( 1+\frac{(N-\ell)(N-\ell-1)}{(N-s)(N-s-1)} \right) ,
\end{align*}
which completes the induction, proving \eqref{eq:loo-stability-general}.
Then \eqref{eq:loo-stability-terminal} follows immediately by setting \(k=N-1\).
\end{proof}

Now we are ready to prove \cref{lemma:S-Dual-OHM-cross-term-expectation-bound}.
Fix \(j \in \{1,\dots,N-1\} \), and let \(x_{N-1}^{(j-1)}\) denote the leave-one-out terminal iterate from \cref{lemma:loo-stability}, obtained by replacing only the minibatch \(\cB_{j-1}\) by an independent copy.
By the fact that \(x_{N-1}^{(j-1)}\) is conditionally independent with $e_{j-1}$ (which depends only on $\cB_{j-1}$) on $\cF_{j-1}$ and the tower property, we have $\expec{\inprod{\opT \big(x_{N-1}^{(j-1)} \big)}{e_{j-1}}}=0$.
Hence 
\begin{align*}
\expec{\inprod{e_{j-1}}{\opT x_{N-1}}} & = \expec{\inprod{e_{j-1}}{\opT(x_{N-1}) - \opT \big(x_{N-1}^{(j-1)}\big) } } \\
& \le \sqrt{
\expec{\sqnorm{e_{j-1}}}
\,
\expec{\sqnorm{\opT(x_{N-1})-\opT\big(x_{N-1}^{(j-1)}\big)}}
} \\
& \le
\sqrt{
\frac{\sigma^2}{B}
\expec{\sqnorm{x_{N-1}-x_{N-1}^{(j-1)}}}
} \\
&\le
\sqrt{
\frac{\sigma^2}{B} \cdot \frac{\alpha^2\sigma^2}{B}
}
=
\frac{\alpha \sigma^2}{B}.
\end{align*}

\subsection{Proof of \texorpdfstring{\cref{lemma:stochastic-trajectory-deviation-bound}}{Lemma 4.6}}

Let
\[
    s_k=\mathbb E\norm{x_k - y_k}^2,
    \qquad
    M_k=\max_{0\le i\le k}s_i .
\]
It suffices to show that $M_k \le \frac{\alpha^2 \sigma^2}{B(1-\gamma^2)}$.
For \(k=0\), we have \(x_0=y_0\), so we have $s_0 = M_0 = 0$ and the result is trivial. 
Let \(k\ge 1\). Applying \cref{lemma:convex-representation} for both \ref{eqn:stochastic-Dual-OHM} and its full-batch version \eqref{eqn:Dual-OHM}, we obtain
\[
    x_k - y_k =
    \sum_{t=0}^{k-1} w_{k,t}
    \left[
        \opT(x_t) - \opT(y_t) + e_t
    \right]
\]
where $e_t = \opT_{\cB_t}(x_t) - \opT(x_t)$ and $w_{k,t} = \begin{cases}
\dfrac{1}{(N-t-1)(N-t)}, & 0\le t\le k-2\\
\dfrac{N-k}{N-k+1}, & t=k-1
\end{cases}$.
Note that here $\condexp{e_t}{\cF_t} = 0$ and $\condexp{\sqnorm{e_t}}{\cF_t} \le \nu^2 := \frac{\alpha^2 \sigma^2}{B}$.
Using \(\sum_{t=0}^{k-1} w_{k,t} = 1-\frac{1}{N} < 1\) and Jensen's inequality gives
\[
\begin{aligned}
    & s_k = \expec{\sqnorm{x_k - y_k}} \le
    \sum_{t=0}^{k-1}w_{k,t}
    \expec{\sqnorm{\opT(x_t)-\opT(y_t)+e_t}} \\
    & \le \sum_{t=0}^{k-1}w_{k,t}
    \left( \expec{\sqnorm{\opT(x_t)-\opT(y_t)} + \sqnorm{e_t}} \right) \le \sum_{t=0}^{k-1} w_{k,t} \left( \gamma^2 \expec{\sqnorm{x_t - y_t}} + \nu^2 \right) 
    \le \gamma^2 M_{k-1} + \nu^2 
\end{aligned}
\]
where for the second inequality, we used conditional unbiasedness together with the tower property:
\[
    \expec{\inprod{e_t}{\opT(x_t) - \opT(y_t)}} = \expec{\condexp{\inprod{e_t}{\opT(x_t) - \opT(y_t)}}{\cF_t}} = 0 
\]
which holds because $\opT(x_t) - \opT(y_t)$ is deterministic conditioned on $\cF_t$ and $\condexp{e_t}{\cF_t} = 0$.
Now starting with $M_0 = 0$ and using induction on $k$, the above shows that
\[
    M_k \le \max\left\{M_{k-1}, \gamma^2 M_{k-1} + \nu^2 \right\} \le \max\left\{M_{k-1}, \gamma^2 \frac{\nu^2}{1-\gamma^2} + \nu^2 \right\} = \frac{\nu^2}{1-\gamma^2} .
\]

\subsection{Proof of \texorpdfstring{\cref{lemma:deterministic-tail-difference-bound}}{Lemma 4.7}}

Observe that for $k=1,2,\dots$, 
\begin{align*}
    & \norm{y_{k+1} - y_k} = \frac{N-k-1}{N-k} \norm{\opT(y_k) - \opT(y_{k-1})} \le \gamma \norm{y_{k} - y_{k-1}} \\
    & \implies \norm{y_{k+1} - y_k} \le \gamma^k \norm{y_1 - y_0} = \gamma^k \norm{\frac{N-1}{N} (\opT(y_0) - y_0)} \le \gamma^k (1+\gamma) \norm{y_0 - x_\star}
\end{align*}
where the last inequality uses $\norm{\opT(y_0) - y_0} \le \norm{\opT(y_0) - x_\star} + \norm{x_\star - y_0} \le (1+\gamma) \norm{y_0 - x_\star}$ (recall that we assume $x_0 = y_0$).
Thus
\begin{align*}
    \norm{y_{N-1} - y_k} &\le \sum_{j=k}^{N-2}\norm{y_{j+1} - y_j} \le (1+\gamma)\norm{y_0 - x_\star} \sum_{j=k}^{N-2} \gamma^j \le \frac{1+\gamma}{1-\gamma} \gamma^k \norm{x_0 - x_\star} .
\end{align*}

\newpage

\section{Additional Experiments}
\label{section:additional-experiments}

In the experiments of Section~\ref{section:experiments}, we used for \algname{S-OHM} the same batch size $B$ selected for \ref{eqn:stochastic-Dual-OHM}. 
This was designed to isolate the effect of replacing the anchor update by the dual-anchor update while
keeping all the other configurations identical. 
However, to verify that the observed empirical advantage of \ref{eqn:stochastic-Dual-OHM} is not an artifact of this experiment design, we repeat all four experiments while separately tuning the batch size of \algname{S-OHM} over the same grid $B\in\{1,10,20,50,100\}$, under the same total sample budget $Q=NB$. 
All the other experimental settings are unchanged.
For simplicity, we only compare the results from \ref{eqn:stochastic-Dual-OHM} with tuned $B$ and runs of \algname{S-OHM} using the shared and tuned batch sizes.

\begin{figure}[ht]
    \centering
    \includegraphics[width=0.47\linewidth]{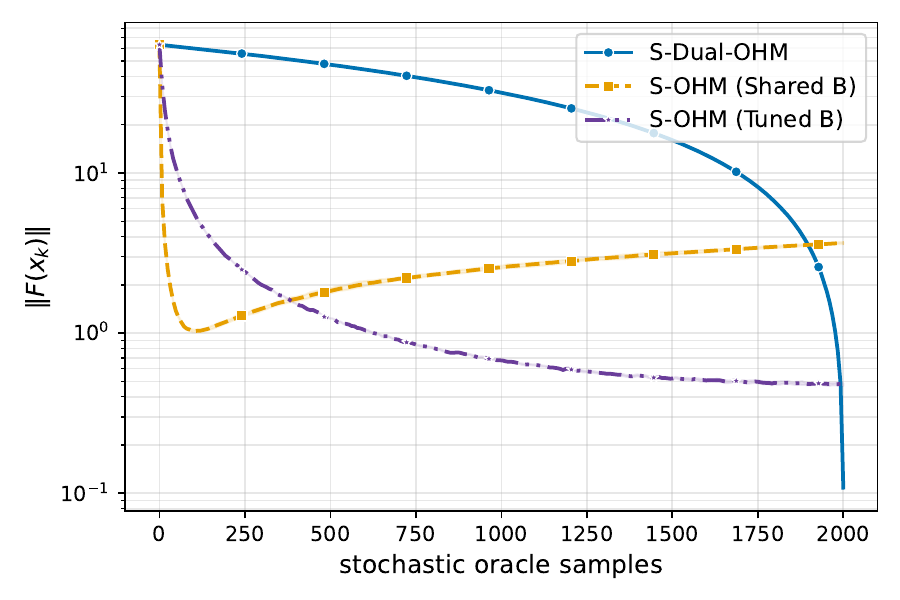}
    \hfill
    \includegraphics[width=0.47\linewidth]{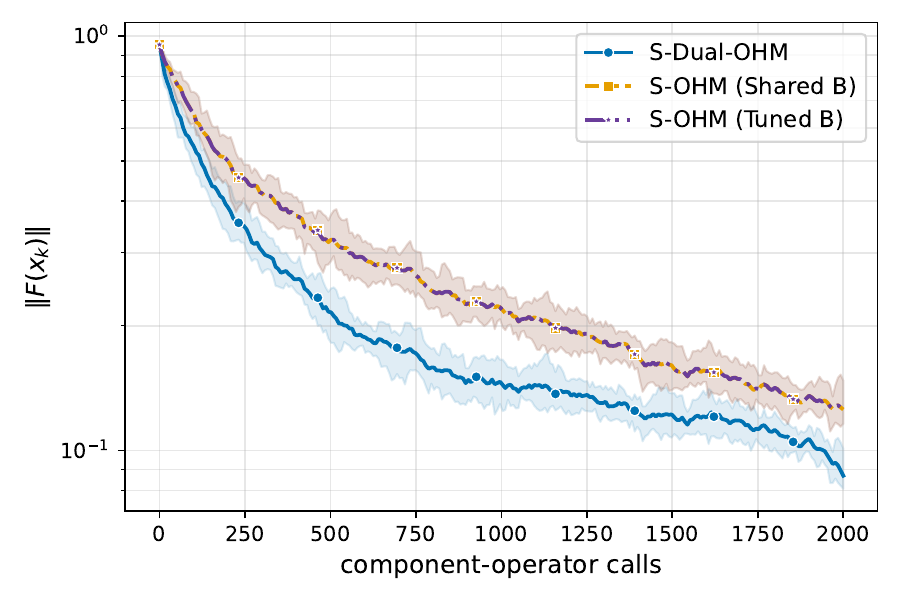} \\
    \includegraphics[width=0.47\linewidth]{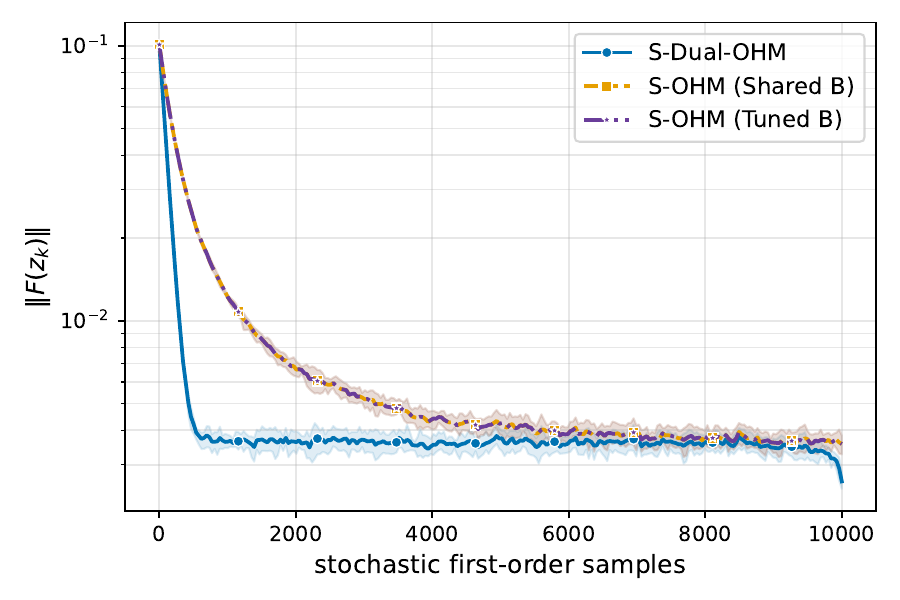}
    \hfill
    \includegraphics[width=0.47\linewidth]{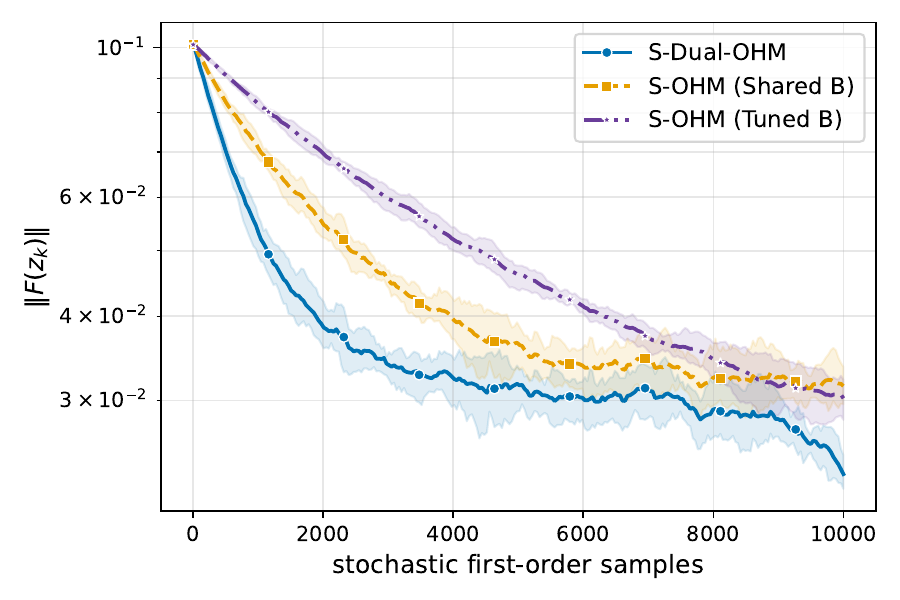}

    \caption{Comparison after independently tuning the constant batch size $B$ of \algname{S-OHM}.
    \textbf{(Top left)} Worst-case nonexpansive operator (Experiment~1).
    \textbf{(Top right)} Finite-sum cocoercive operator (Experiment~2).
    \textbf{(Bottom left)} SCSC Huber-type minimax problem (Experiment~3), in the low-variance regime with $\sigma=0.05$.
    \textbf{(Bottom right)} The same minimax problem from Experiment~3 in the high-variance
    regime with $\sigma=1.5$.
    The plots labeled as \algname{S-OHM} (Tuned $B$) use the best $B\in\{1,10,20,50,100\}$ selected under the given total sample budget. 
    Solid curves are means over 10 independent runs, and shaded bands denote empirical
    5th--95th percentiles.}
    \label{figure:additional-experiments}
\end{figure}

As shown in Figure~\ref{figure:additional-experiments}, independent
tuning selects the same batch size for \algname{S-OHM} and
\ref{eqn:stochastic-Dual-OHM} in Experiment~2 and the low-variance regime in Experiment~3.
In Experiment~1 and the high-variance regime in Experiment~3, the batch sizes independently chosen for \algname{S-OHM} differ.
Nevertheless, even with its own optimized batch size, the final residual attained by \algname{S-OHM} remains strictly larger than that of \ref{eqn:stochastic-Dual-OHM}. 
These results from controlled experiments further supports the empirical benefit/robustness of the dual-anchor mechanism over anchoring in stochastic settings.


\end{document}